\RequirePackage{fix-cm}
\documentclass[twocolumn,numbook,natbib]{svjour3}          
\smartqed  
\usepackage{graphicx,amssymb,mathtools,color}
\usepackage[hyphens]{url}
\usepackage{breakurl}
\usepackage[hidelinks,breaklinks]{hyperref}

\definecolor{darkred}{rgb}{.7,0,0}

\definecolor{darkgreen}{rgb}{0,0.7,0}

\definecolor{darkblue}{rgb}{0,0,0.7}

\newcommand*{\defeq}{\coloneqq}
\newcommand*{\qefed}{\eqqcolon}

\newcommand*{\Expect}{\mathbb{E}}

\newcommand*{\Lip}{\mathop{\mathrm{Lip}}\nolimits}
\newcommand*{\Naturals}{\mathbb{N}}
\newcommand*{\odeflow}{\Phi}
\newcommand*{\numflow}{\Psi}
\newcommand*{\ProbSpace}{\Omega}
\newcommand*{\quark}{\setbox0\hbox{$x$}\hbox to\wd0{\hss$\cdot$\hss}}
\newcommand*{\rd}{\mathrm{d}}
\newcommand*{\Reals}{\mathbb{R}}
\newcommand*{\tstep}{\tau}
\newcommand*{\norm}[1]{\Vert #1 \Vert}
\newcommand*{\innerprod}[2]{\langle #1 , #2 \rangle}
\newcommand*{\absval}[1]{\vert #1 \vert}
\newcommand*{\Norm}[1]{\left\Vert #1 \right\Vert}
\newcommand*{\Innerprod}[2]{\left\langle #1 , #2 \right\rangle}
\newcommand*{\Absval}[1]{\left\vert #1 \right\vert}

\newtheorem{assumption}[theorem]{Assumption}

\journalname{Statistics and Computing}
\begin{document}

\title{Strong convergence rates of probabilistic integrators for ordinary differential equations 
\thanks{HCL and TJS are supported by the Freie Universit\"{a}t Berlin within the Excellence Initiative of the German Research Foundation. HCL is supported by the Universit\"{a}t Potsdam.
AMS is grateful to DARPA, EPSRC and ONR for funding. 
This material was based upon work partially supported by the National Science Foundation under Grant DMS-1127914 to the Statistical and Applied Mathematical Sciences Institute.
Any opinions, findings, and conclusions or recommendations expressed in this material are those of the author(s) and do not necessarily reflect the views of these funding agencies and institutions.}
}

\titlerunning{Strong convergence of probabilistic integrators}        

\author{Han Cheng Lie
 \and
         A. M. Stuart
         \and 
         T. J. Sullivan 
}

\authorrunning{Lie, Stuart, and Sullivan} 

\institute{Han Cheng Lie \at
              Institute of Mathematics, Universit\"{a}t Potsdam, Campus Golm, Haus 9, Karl-Liebknecht Str. 24-25, 14476 Potsdam OT Golm, Germany\\
              \url{https://orcid.org/0000-0002-6905-9903}\\
              \email{hanlie@uni-potsdam.de}           
           \and
              A. M. Stuart\at
              Department of Computing and Mathematical Sciences, California Institute of Technology, 1200 East California Boulevard, Pasadena, CA 91125, United States of America\\
              \email{astuart@caltech.edu}
              \and
           T. J. Sullivan\at
              Freie Universit\"{a}t Berlin, and Zuse Institute Berlin, Takustrasse 7, 14195 Berlin, Germany\\
              \email{sullivan@zib.de}
}

\date{Received: September 28, 2018 / Accepted: December 28, 2018 / Handling Editor: C.\ J.\ Oates \vskip1ex This is a post-peer-review, pre-copyedit version of an article published in Statistics and Computing (2019). The final authenticated version is available online at: \url{https://doi.org/10.1007/s11222-019-09898-6}.}

\maketitle

\begin{abstract}
Probabilistic integration of a continuous dynamical system is a way of systematically introducing discretisation error, at scales no larger than errors introduced by standard numerical discretisation, in order to enable thorough exploration of possible responses of the system to inputs.
It is thus a potentially useful approach in a number of applications such as forward uncertainty quantification, inverse problems, and data assimilation.
We extend the convergence analysis of probabilistic integrators for deterministic ordinary differential equations, as proposed by Conrad et al.\ (\textit{Stat.\ Comput.}, 2017), to establish mean-square convergence in the uniform norm on discrete- or continuous-time solutions under relaxed regularity assumptions on the driving vector fields and their induced flows.
Specifically, we show that randomised high-order integrators for globally Lipschitz flows and randomised Euler integrators for dissipative vector fields with polynomially-bounded local Lipschitz constants all have the same mean-square convergence rate as their deterministic counterparts, provided that the variance of the integration noise is not of higher order than the corresponding deterministic integrator. These and similar results are proven for probabilistic integrators where the random perturbations may be state-dependent, non-Gaussian, or non-centred random variables.

\end{abstract}

\section{Introduction}
\label{sec:introduction}

This article concerns the analysis of probabilistic numerical integrators for deterministic initial value problems of the form
\begin{align}
	\label{eq:the_ivp}
	\frac{\mathrm{d}}{\mathrm{d} t} u(t)  &= f(u(t)), \quad \text{for $0 \leq t \leq T$,}
	\\
	u(0) &= u_{0}\in\Reals^{d}
	\notag
\end{align}
where $T > 0$. Let $\odeflow^{t} \colon \Reals^{d} \to \Reals^{d}$ denote the flow induced by \eqref{eq:the_ivp}, so that 
\begin{equation}
	\label{eq:odeflow}
	\odeflow^{t}(u_{0}) = u_{0} + \int_{0}^{t} f(\odeflow^{s}(u_{0}) ) \, \rd s
\end{equation}
for all $(t, u_{0}) \in [0, T] \times \Reals^{d}$. Given an integration time step $\tstep > 0$ such that $K \defeq T / \tstep \in \Naturals$ and the corresponding time mesh
\begin{equation}
	\label{eq:time_grid}
	t_{k} \defeq k \tstep \text{ for } k \in [K] \defeq \{ 0, 1, \dotsc, K \} ,
\end{equation}
a deterministic one-step numerical method for the solution of \eqref{eq:the_ivp} is a numerical flow map $\numflow^{\tstep} \colon \Reals^{d} \to \Reals^{d}$ that generates approximations $\widetilde{u}_{k} \approx u_{k} \defeq u(t_{k})$ by the recursion $\widetilde{u}_{k} \defeq \numflow^{\tstep}(\widetilde{u}_{k-1})$;
note that $u_{k} = \odeflow^{\tstep}(u_{k - 1})$.
A key property of the numerical method is its global order of convergence, i.e.\ the largest $q \geq 0$ such that, for some constant $C=C(T)$, independent of $\tstep$,
\begin{equation}
	\label{eq:model_for_convergence}
	\max_{k\in[K]} \norm{ u_{k} - \widetilde{u}_{k} }\leq C \tstep^{q}.
\end{equation}
As a modelling choice, epistemic stochasticity can be introduced into the numerical solution of \eqref{eq:the_ivp} on the basis that, while the exact solution satisfies 
\begin{equation}
	\label{eq:integrand_g}
	u_{k+1} = \odeflow^{\tstep}(u_{k}) = u_{k} + \int_{t_{k}}^{t_{k+1}} f(u(s)) \, \rd s
\end{equation}
for all $k \in [K-1]$, the only information available about the values of the solution off the time mesh comes from the numerical solution on the mesh, and so the integrand $f(u(s))$ is not exactly accessible.
This uncertainty is relevant in the setting where, given a large-scale mathematical model, it may be more statistically informative to spend computational resources on solving a differential equation-based model many times on a coarser grid than on solving the same model a few times on a finer grid.
This is often the case in forward uncertainty quantification \citep{Smith:2014,Sullivan:2015}, inverse problems \citep{KaipioSomersalo:2005,Stuart:2010}, and data assimilation \citep{LawStuartZygalakis:2015,ReichCotter:2015};
the area of multi-level Monte Carlo methods makes particular use of this kind of cost-accuracy tradeoff \citep{Giles:2015}.
Furthermore, in many such settings, the quantity of interest is often not the solution of a differential equation-based model, but a functional thereof.
In all cases, estimates of the off-mesh uncertainty due to the numerical method can and should be fed forward to estimate the uncertainty in the quantity of interest.

This article is motivated by the work of \citet{Conrad:2016}, in which one seeks to model the off-mesh uncertainty by considering probabilistic solvers.
For the same mesh given in \eqref{eq:time_grid}, the probabilistic solver of \citet{Conrad:2016} involves producing a sequence of random variables $(U_{k})_{k\in [K]}$ according to
\begin{equation}
	\label{eq:randomised_numerical_scheme_0}
	U_{k+1} \defeq \numflow^{\tstep} (U_{k}) + \xi_{k}(\tstep), \quad U_{0}=u_{0},
\end{equation}
where $\numflow^{\tstep}$ is the map associated to the deterministic numerical method, and each $\xi_{k}(\tstep)$ is an i.i.d.\ copy of a random variable $\xi_{0}(\tstep) \defeq \int_{0}^{\tstep} \chi_{0}(s) \, \rd s$, where $\chi_{0}$ is a stochastic process over the time interval $[0, \tstep]$ that models the off-mesh behaviour of the unknown function $f(u(s))$ in \eqref{eq:integrand_g}. We refer the reader to \citet[Figure 2]{Conrad:2016} for a pictorial representation of \eqref{eq:randomised_numerical_scheme_0}.
The process $\chi_{0}$ is introduced so that one can probe the uncertainty induced by the mesh $(t_{k})_{k\in[K]}$ and the underlying solver, and thus explore possible responses of the system to inputs. Comparing \eqref{eq:integrand_g} and \eqref{eq:randomised_numerical_scheme_0}, it follows that the random variable $\xi_k(\tstep)$ is a statistical model for the approximation error $\odeflow^{\tstep}(u_k)-\numflow^{\tstep}(u_k)$.

We emphasise that the additive, state-independent noise model appearing in \eqref{eq:randomised_numerical_scheme_0} should be interpreted as providing a prior on the local truncation error \citep{HairerNorsettWanner:2009}.
A frequent criticism levelled at the field of probabilistic numerical integration is that the statistical properties of the noise $\xi_{k}$ that have been imposed in existing published works do not reflect known prior information about local truncation error.
Here we address this issue by considerably weakening the assumptions made on the $\xi_{k}$.
However we anticipate future work in this direction, especially when specific structure on the vector field $f$ is used to further inform the prior.
Note also that, in the presence of large amounts of data, we expect posterior contraction and forgetting of the prior;
see, e.g., \citet{Knapik:2011}.
Posterior contraction for \eqref{eq:randomised_numerical_scheme_0} was demonstrated numerically on a number of examples by \citet{Conrad:2016}.

In the spirit of \eqref{eq:model_for_convergence}, the main convergence result \citep[Theorem 2.2]{Conrad:2016} yields that, if the vector field $f$ in \eqref{eq:the_ivp} is globally Lipschitz, if the deterministic numerical method has uniform local truncation of order $q+1$, and if $\chi_{0}$ is a centred Gaussian process such that the second moment of $\xi_{0}(\tstep)$ decays as $\tstep^{2 p + 1}$ for some $p \geq 1$, then
\begin{equation}
	\label{eq:intro_convergence_sup_outside}
	\max_{ k\in [K]}\Expect \bigl[ \norm{u_{k}-U_{k}}^2 \bigr]\leq C \tstep^{2 \min\{p, q\}}.
\end{equation}
This shows that the convergence rate of the probabilistic solver \eqref{eq:randomised_numerical_scheme_0} is determined by the convergence of the `worst-case error' of the deterministic method $\numflow^{\tstep}$, and the convergence of the `statistical error' $\xi_{0}$, as described by the parameters $q$ and $p$ respectively.
Choosing $\xi_{0}$ with $p = q$ introduces the maximum amount of solution uncertainty consistent with preserving the order of accuracy of the original deterministic integrator.

It is important to stress that, despite the apparent similarities between
\eqref{eq:randomised_numerical_scheme_0} and Euler--Maruyama schemes for stochastic differential equations (SDEs) driven by Brownian motion, the analysis of the latter does not directly apply to probabilistic solvers, even though we will borrow some techniques from that field.
This is because the variance of $\xi_{0}(\tstep)$ for probabilistic solvers of the form \eqref{eq:randomised_numerical_scheme_0} is assumed to decay to zero strictly faster than $\tstep$, whereas, for SDEs driven by Brownian motion, the variance is proportional to $\tstep$.
A key aspect of this work is to determine how to scale the noise so that the rate of convergence of the underlying deterministic numerical integrator is not affected, yet uncertainty arising from numerical approximation is accounted for.

\subsection{Contribution and outline of the paper}

The purpose of this paper is to make significant extensions of the convergence analysis of \citet{Conrad:2016} for \eqref{eq:the_ivp}.
We accomplish this by obtaining stronger error estimates (and hence stronger convergence results) under assumptions on both the underlying differential equation and on the noise model for probabilistic numerical integration that are weaker than their counterparts in \citet{Conrad:2016}.
The convergence results of this paper are of the form
\begin{equation}
	\label{eq:intro_convergence_sup_inside}
	\Expect \biggl[ \max_{k \in [K]} \norm{ u_{k} - U_{k} }^{n} \biggr] \leq C \tstep^{n\cdot\min\{ p-c, q\}},
\end{equation}
where $n\in\Naturals$, $q$ is the order of the numerical method $\numflow^{\tstep}$, $p$ is an exponent of decay in the moments of the random variables $(\xi_{k}(\tstep))_{k\in[K]}$, and $c\geq 0$ is a penalty term in the convergence rate that depends solely on the random variables $(\xi_{k}(\tstep))_{k\in[K]}$.
Note that, when $c=0$ and $n=2$, the convergence rate of $n\min\{p-c,q\} $ on the right-hand side of \eqref{eq:intro_convergence_sup_inside} agrees with that of \eqref{eq:intro_convergence_sup_outside} shown by \citet{Conrad:2016}, so that the right-hand sides of \eqref{eq:intro_convergence_sup_outside} and \eqref{eq:intro_convergence_sup_inside} differ only in the constant prefactor $C$.
However, because the time supremum is inside the expectation, \eqref{eq:intro_convergence_sup_inside} implies \eqref{eq:intro_convergence_sup_outside}.
Furthermore, by Markov's inequality, \eqref{eq:intro_convergence_sup_inside} yields an estimate of the frequentist coverage of the true solution $u$ by the randomised solutions $U$:
\begin{equation*}
	\mathbb{P} \biggl[ \max_{k \in [K]} \norm{ u_{k} - U_{k} } \leq r \biggr] \geq 1 - C \tstep^{n\min\{ p-c,q\}} r^{-n} ;
\end{equation*}
such estimates are useful in the context of forward uncertainty quantification and inverse problems \citep{LieSullivanTeckentrup:2018}.

We emphasize that, in addition to strengthening the form of the convergence results so that the supremum is inside the expectation, we also prove the results in this paper under weaker assumptions on the vector field $f$, and under weaker assumptions on the
noise $\xi_{k}$, than those employed by \citet{Conrad:2016}.
Specifically
we do not assume that $f$ and its derivatives are globally bounded, and we do not assume that the random variables are Gaussian;
furthermore in results generalizing \eqref{eq:intro_convergence_sup_inside} we relax the assumption that the random variables are centred, paving the way for future analyses which incorporate specific known structure and bias in the truncation error.

Error estimates like \eqref{eq:intro_convergence_sup_inside} show that the randomised numerical solution has convergence properties that are asymptotically no worse than the deterministic numerical solution.
This can be interpreted as saying that the trajectories obtained from the randomised numerical integrator are all equally valid approximations to the solution of the original system, modulo the uncertainty induced by solving in discrete time.
This can be useful for many purposes, for example in studying limits on predictability in chaotic systems, as shown for the Lorenz-63 system by \citet{Conrad:2016}.

After introducing some notation and auxiliary results in Section~\ref{sec:setup}, the rest of the paper is organised as follows.
In Section~\ref{sec:lipschitz}, Theorem~\ref{thm:strong_error_sup_inside} yields \eqref{eq:intro_convergence_sup_inside} for numerical methods of arbitrary order, for vector fields $f$ whose induced flow maps $\odeflow^{\tstep}$ are globally Lipschitz --- including one-sided Lipschitz vector fields --- and for collections $(\xi_{k}(\tstep))_{k\in[K]}$ of random variables that are independent and centred, but not necessarily Gaussian.
\citet{Conrad:2016} assumed the vector field $f$ to be globally Lipschitz, and the random variables $(\xi_{k}(\tstep))_{k\in[K]}$ were assumed to be i.i.d.\ centred Gaussian random variables.
In Theorem~\ref{thm:convergence_without_iid_mean_zero_noise}, we prove a result similar to \eqref{eq:intro_convergence_sup_inside} in which we relax the assumption that the $(\xi_{k}(\tstep))_{k\in[K]}$ are independent and that they are centred;
the price we pay for these weaker constraints on the noise is a stronger decay assumption, with respect to the time-step, on the second moments of the $(\xi_{k}(\tstep))_{k\in[K]}$. We use this assumption in order to introduce the maximal noise that is consistent with retaining the rate of convergence of the underlying deterministic numerical integrator.

In Section~\ref{sec:dissipative}, we further weaken the conditions on the vector field $f$, by considering locally Lipschitz vector fields that satisfy a polynomial growth condition.
In Theorem~\ref{thm:strong_error_sup_inside_polynomial_growth_and_bounded_noise}, we show that, under the assumption that the $(\xi_{k}(\tstep))_{k\in[K]}$ are almost surely bounded, we can again obtain \eqref{eq:intro_convergence_sup_inside}.
In Theorem~\ref{thm:convergence_without_iid_mean_zero_noise_implicit_euler}, we remove the almost-sure boundedness condition, but add the assumption that the vector field $f$ satisfies a generalised dissipativity condition.

In Section~\ref{sec:additional_results} we discuss a continuous-time analogue of \eqref{eq:randomised_numerical_scheme_0}, and show how convergence results of the form \eqref{eq:intro_convergence_sup_inside} can be obtained.
We also show that there exists a nonempty set of random variables (or more generally, stochastic processes) that satisfy the regularity assumptions on the random variables $(\xi_{k}(\tstep))_{k\in[K]}$ used throughout this paper. 

Proofs of the results may be found in Appendix \ref{sec:appendix}.

\subsection{Review of probabilistic numerical methods}

Continuous relationships such as ODEs and PDEs are commonplace as forward models in uncertainty quantification problems, or as Bayesian likelihoods in modern statistical inverse problems \citep{KaipioSomersalo:2005,Stuart:2010}, and in particular in data assimilation algorithms with critical everyday applications such as numerical weather prediction \citep{LawStuartZygalakis:2015,ReichCotter:2015}.
The use of a discretised solver for such forward models is usually unavoidable in practice, but introduces an additional source of uncertainty both into forward propagation of uncertainty and into subsequent inferences.
While the solution to the ODE/PDE may not be random in the frequentist sense, it is nonetheless only imperfectly known through the discretised numerical solution.
Probability in the subjective or Bayesian sense is one appropriate means of representing this epistemic uncertainty, particularly if the ODE/PDE solution forms part of the forward model in a Bayesian inverse problem.
Failure to properly account for discretisation errors and uncertainties can lead to biased, inconsistent, and over-confident inferences \citep{Conrad:2016}.

Probabilistic numerical solutions of problems such as the solution of ODEs have a long history.
Modern foundations for this field were laid by the work of \citet{Diaconis:1988}, \citet{OHagan:1992}, and \citet{Skilling:1992} under the term of ``Bayesian numerical analysis''.
More recently, such ideas have received renewed attention under the term ``probabilistic numerics'' \citep{HennigOsborneGirolami:2015,Cockayne:2017b}:
the discussion of probabilistic numerical methods for ordinary differential equations given by \citet{Schober:2014,Conrad:2016,Chkrebtii:2016}, and \citet{Teymur:2018} is particularly relevant here.
Also of interest in the field of probabilistic numerics, but not directly relevant to the present work, are probabilistic numerical methods for linear algebra \citep{Hennig:2015}, optimisation \citep{Gonzalez:2016}, partial differential equations \citep{Cockayne:2017a,Owhadi:2015,Owhadi:2017,Wang:2018}, and quadrature \citep{Briol:2015}.
In particular, \citet{Cockayne:2017b} sets out some axiomatic foundations for probabilistic numerical methods broadly conceived, and in particular what it means for a probabilistic numerical method to be ``Bayesian''.

Randomised solutions of ODEs have also been studied in the context of stochastic or rough differential equations.
In the case of non-autonomous ODEs driven by Carath{\'e}odory vector fields --- i.e.\ vector fields that are locally integrable in time and continuous in the state space --- it has been observed that randomised Euler and Runge--Kutta methods outperform their deterministic counterparts: see e.g.\ \citet{Stengle:1990,JentzenNeuenkirch:2009}, and \citet{KruseWu:2017} and the references therein.

We note that analysing the convergence properties of numerical solutions to \eqref{eq:the_ivp} in terms of the approximation error for the solution, as in \eqref{eq:intro_convergence_sup_outside} and \eqref{eq:intro_convergence_sup_inside}, is very much in the spirit of classical numerical analysis.
For uncertainty quantification of the discretised solution of \eqref{eq:the_ivp} as a stand-alone forward problem, this viewpoint is often sufficient.
However, for applications to inverse problems and data assimilation, in which the numerical solution of the \eqref{eq:the_ivp} is used to (approximately) evaluate the data misfit or likelhood, an alternative paradigm is to directly examine the impact of discretisation upon the quality of later inferences using e.g.\ Bayes factors \citep{CapistranChristenDonnet:2016,Christen:2017}.
There is also the well-established literature of information-based complexity and average-case analysis, with its greater emphasis on algorithmic aspects such as computational cost and optimal accuracy for given classes of information \citep{Novak:1988,Ritter:2000,TraubWozniakowski:1980,TraubWasilkowskiWozniakowski:1983}.

\section{Setup and notation}
\label{sec:setup}



Let $(\ProbSpace, \mathcal{F}, \mathbb{P})$ be a probability space sufficiently rich to serve as a common domain of definition for all the random variables and processes under consideration, and let $\Expect$ denote expectation with respect to $\mathbb{P}$.
The space of $s$\textsuperscript{th}-power integrable random variables over $(\ProbSpace, \mathcal{F}, \mathbb{P})$ will be denoted $L^{s}_{\mathbb{P}}$.
The scalars $C$, $C'$, etc.\ denote non-negative constants whose value may change from occurence to occurence, but are independent of the time step $\tstep > 0$.
$\Lip(\Phi)$ denotes the best Lipschitz constant of $\Phi \colon
\Reals^{d}\to\Reals^{d}$:
\[
	\Lip(\Phi) \defeq \min \{L \geq 0 \mid \norm{ \Phi(x) - \Phi(y) } \leq L \norm{ x - y } \}
\]
for all $x, y \in \Reals^{d}$.
We let $\Naturals$ denote the natural numbers beginning with $1$, and
$\Naturals_{0} \defeq \Naturals \cup \{ 0 \}$.
We shall sometimes abuse notation and write $[K] \defeq \{0, 1, \dotsc, K - 1 \}$
or $[K] \defeq \{ 1, 2, \dotsc, K \}$, and we shall write $u_{k} \defeq u(t_{k}) \equiv \odeflow^{\tstep}(u_{k - 1})$ for the value of the exact solution to \eqref{eq:the_ivp} at time $t_{k}$.
We denote the minimum of a pair of real numbers $a$ and $b$ by $a\wedge b=\min\{a,b\}$.

It will be assumed throughout that $T > 0$ is a fixed, deterministic time, and that $f$ in \eqref{eq:the_ivp} is sufficiently smooth such that \eqref{eq:the_ivp} has a unique solution for every initial condition $u_0$.
The flow map $\odeflow^{t}$ associated to \eqref{eq:the_ivp} is defined in \eqref{eq:odeflow}, and the output of a one-step deterministic numerical integration method for a given $x$ and time step $\tstep$ will be given by $\numflow^{\tstep}(x)$.
This setting encompasses many of the time-stepping methods in common use, such as Runge--Kutta methods of all orders.

%

The analysis of this paper will make repeated use of several useful inequalities, which are collected here for reference.
First, recall Young's inequality:
for any $\delta > 0$ and any pair of H\"{o}lder conjugate exponents $r,r^\ast>1$,
\begin{equation}
	\label{eq:Young}
	ab \leq \frac{\delta}{r} a^r + \frac{1}{r^\ast \delta^{r^\ast/r}} b^{r^\ast},\quad\text{for all $a,b\geq 0$.}
\end{equation}
Combining that inequality for $r=r^\ast=2$ with the Cauchy--Schwarz inequality in $\Reals^{d}$ yields
\begin{equation}
	\label{eq:conseq_of_Young}
	\norm{ x - y }^{2} \leq (1 + \delta) \norm{ x }^{2} + (1 + \delta^{-1}) \norm{ y }^{2} ,
\end{equation}
which will often be used either with $\delta = 1$ or $\delta = \tstep$.

The following discrete-time version of Gr{\"o}nwall's inequality \citep{Holte:2009} will also be useful:
\begin{theorem}
	\label{thm:discrete_Gronwall}
	Let $(x_{k})_{k \in \Naturals_{0}}$, $(\alpha_{k})_{k \in \Naturals_{0}}$, and $(\beta_{k})_{k \in \Naturals_{0}}$ be non-negative sequences.
	If, for all $k \in \Naturals_{0}$,
	\begin{equation*}
		x_{k} \leq \alpha_{k} + \sum_{0 \leq j < k} \beta_{j} x_{j}
		\quad
		\text{ and }
		\quad
		\alpha_{k} \leq A,
	\end{equation*}
	then $x_{k} \leq A \exp \left( \sum_{0 \leq j < k} \beta_{j} \right)$ for all $k \in \Naturals_{0}$.
\end{theorem}

For completeness, we state the following lemma. 

\begin{lemma}
	\label{lem:gen_peter_paul_inequality}
	Let $x,y\geq 0$, $n\in\Naturals$, and $\delta>0$.
	Then
	\begin{equation}
		\label{eq:gen_peter_paul_inequality}
		(x + y)^n\leq x^n ( 1 + \delta 2^{n-1} ) + y^n (1 + (2/\delta)^{n-1} ).
	\end{equation}
\end{lemma}

We shall also use the following inequality, which is valid for arbitrary $N \in \Naturals$ and $m\geq 1$:
for all $\{s_j\}_{j\in [N]}\in \Reals^{N}$,
\begin{equation}
	\Absval{ \sum^N_{j=1}s_j }^m \leq N^{m-1}\sum^N_{j=1} \absval{ s_j }^{m} .
	\label{eq:gen_triangle_inequality}
\end{equation}
This follows from
 \begin{align*}
  \Absval{ \sum^N_{j=1}s_j }^m &\leq N^{m}\left(\frac{1}{N}\sum^N_{j=1}\absval{ s_j }\right)^{m}\leq N^{m}\left(\frac{1}{N}\sum_{j=1}^N\absval{ s_j }^{m}\right)
  \\
  &=N^{m-1}\sum^N_{j=1} \absval{ s_j }^{m},
 \end{align*}
where we used Jensen's inequality in the second inequality.

\section{High-order integration of Lipschitz flows}
\label{sec:lipschitz}


The purpose of this section is to establish, given the initial value problem \eqref{eq:the_ivp}, the strong convergence result \eqref{eq:intro_convergence_sup_inside} for probabilistic solvers of the form \eqref{eq:randomised_numerical_scheme_0}, under the following assumptions.

\begin{assumption}
	\label{ass:exact_flow}
	The vector field $f$ admits $0<\tstep^\ast\leq 1$ and $C_\odeflow\geq 1 $, such that for $0< \tstep<\tstep^\ast$, the flow map $\odeflow^{\tstep}$ defined by \eqref{eq:odeflow} is globally Lipschitz with Lipschitz constant $\Lip(\odeflow^{\tstep}) \leq 1 + C_\odeflow \tstep$.
\end{assumption}

As is well known, Assumption \ref{ass:exact_flow} holds if the generating vector field $f$ is itself globally Lipschitz.
However, Assumption \ref{ass:exact_flow} holds if, for instance, $f$ merely satisfies the one-sided Lipschitz inequality
\begin{equation*}
	\innerprod{ f(x) - f(y) }{ x - y } \leq \mu \norm{ x - y }^{2},\quad\text{for all $x, y\in\Reals^{d}$,}
\end{equation*}
for some constant $\mu \in \Reals$;
in this case, a calculation of $\frac{\rd}{\rd t} \norm{ u(t) - v(t) }^{2}$ for trajectories $u$ and $v$ starting at initial conditions $u_{0}, v_{0} \in \Reals^{d}$ and an application of the differential version of Gr\"onwall's inequality shows that $\norm{ u(t) - v(t) } \leq \exp( \mu \absval{ t } ) \norm{ u_{0} - v_{0} } $, so that $\Lip(\odeflow^{t}) \leq 1 + 2 \absval{ \mu } \absval{ t }$ for small $\absval{ t }$.


\begin{assumption}
	\label{ass:numerical_flow}
	The numerical method $\numflow^{\tstep}$ has uniform local truncation error of order $q + 1$:
	for some constant $C_\numflow\geq 1$ that does not depend on $\tstep$,
	\[
		\sup_{u \in \Reals^{d}} \norm{ \numflow^{\tstep}(u) - \odeflow^{\tstep}(u) } \leq C_\numflow \tstep^{q + 1}.
	\]
\end{assumption}

Assumption \ref{ass:numerical_flow} holds, in particular, for single- and multi-step methods derived from a $q$-times continuously differentiable vector field $f$ with bounded $q$\textsuperscript{th} 
derivatives \citep[Section~III.2]{HairerNorsettWanner:2009}.
Imposing global bounds on the derivatives of $f$, and therefore on those of $\odeflow^\tstep$, forces us to consider a smaller class of flow maps $\odeflow^\tstep$ than the class of flow maps that satisfy Assumption \ref{ass:exact_flow}.
We may alleviate this problem by weakening Assumption \ref{ass:numerical_flow} to a bound of the form
\begin{equation}
	\label{eq:weaker_asmp_on_numerical_flow}
	\norm{\odeflow^\tstep(u)-\numflow^\tstep(u)}\leq C'(u)\tstep^{q + 1},
\end{equation}
with the consequence that the dependence of $C'(u)$ on $u$ must be specified;
this dependence will vary according to the chosen numerical method $\numflow^\tstep$.
Moreover, whenever we apply \eqref{eq:weaker_asmp_on_numerical_flow} in place of Assumption \ref{ass:numerical_flow} with a random variable $U_k$ in place of a deterministic $u_k$ --- as we do below, e.g.\ in deriving \eqref{eq:estimate01} --- we will need to ensure that $\Expect[C'(U_k)]$ is finite, and of the correct order in $\tstep$ if necessary.
In Section~\ref{sec:dissipative}, we consider the implicit Euler method for a class of \textit{locally} Lipschitz flow maps $\odeflow^\tstep$, obtain an expression for $C'(U_k) $, and with this expression obtain a bound of the form
\begin{equation*}
	\Expect \bigl[ \norm{\odeflow^\tstep(U_k)-\numflow^\tstep(U_k)}^{n} \bigr] \leq C \tstep^{n(q + 1)}
\end{equation*}
where $U_k$ denotes the output of the randomised numerical integrator according to \eqref{eq:randomised_numerical_scheme_0}, $n \in \Naturals$, and $C>0$ does not depend on $\tstep$ or on $k$;
see Proposition \ref{prop:W_k}.
Note that there is no supremum inside the expectation in the inequality above.
However, in this section, we shall apply Assumption \ref{ass:numerical_flow} instead of \eqref{eq:weaker_asmp_on_numerical_flow}, in order to avoid lengthy analyses that are specific to the choice of numerical method.
We make no assumptions about how the integrator $\numflow^{\tstep}$ has been derived and treat it as a `black box' satisfying Assumption \ref{ass:numerical_flow}.

\begin{assumption}
	\label{ass:p-R_regularity_condition_on_noise}
	The random variables $(\xi_{k}(\tstep))_{k\in\Naturals}$
	admit parameters $p \geq 1$, $R\in\Naturals\cup\{ + \infty\}$, and $C_{\xi,R}\geq 1$, independent of $k$ and $\tstep$, such that for all $1\leq r\leq R$ and all $k\in\Naturals$,
	\begin{equation*}
		\Expect \bigl[ \norm{\xi_{k}(\tstep)}^{r} \bigr] \leq \left( C_{\xi,R} \tstep^{p + 1/2} \right)^r.
	\end{equation*}	
\end{assumption}

Note that we do not assume that the $(\xi_k(\tstep))_{k\in[K]}$ are identically distributed nor that they are centred.
However we will impose these two additional assumptions in Theorem \ref{thm:strong_error_sup_inside}.
The parameter $p$ determines the decay rate of the $r$\textsuperscript{th} moments of the $(\xi_{k}(\tstep))_{k\in[K]}$, for $1\leq r\leq R$, while $R$ determines the highest order moment for which the same decay behaviour holds.

Since Assumption \ref{ass:p-R_regularity_condition_on_noise} does not assume that the $\xi_{k}$ are identically distributed or mutually independent, it can hold for the following variant of \eqref{eq:randomised_numerical_scheme_0}:
\begin{equation*}
	U_{k + 1} \defeq \numflow^{\tstep}(U_{k}) + \xi_{k}(\tstep,U_{k}),\quad\text{for all $k\in[K]$.}
\end{equation*}
In this setting, we interpret Assumption \ref{ass:p-R_regularity_condition_on_noise} as the condition that the dependence of the moments of $\xi_k$ on the state $U_k$, can be uniformly controlled by the constant $C_{\xi,R}$.
We leave a more extensive investigation of state-dependent noise models for future work.

It follows from \eqref{eq:gen_triangle_inequality} and Assumption \ref{ass:p-R_regularity_condition_on_noise} that, for $v,w\in\Naturals$,
\begin{align}
	\Expect\left[\left(\sum^{T/\tstep}_{i=1} \norm{\xi_i(\tstep)}^w \right)^v\right]
	& \leq \left(TC_{\xi,R}^w\tstep^{w(p + 1/2)-1}\right)^v.
	\label{eq:upper_bound_expectation_nth_power_sum_norm_xi_sq}
\end{align}
This is because
 \begin{align*}
  \Expect\left[\left(\sum^{T/\tstep}_{i=1} \norm{\xi_i(\tstep)}^w \right)^v\right]&\leq \left(\frac{T}{\tstep}\right)^{v-1}\sum^{T/\tstep}_{i=1} \left(C_{\xi,R} \tstep^{p + 1/2}\right)^{wv}
  \\
  &\leq \left(\frac{T}{\tstep}\right)^{v}\left(C_{\xi,R}\tstep^{p + 1/2}\right)^{wv}
  \\
  &=\left(TC_{\xi,R}^w\tstep^{w(p + 1/2)-1}\right)^v,
 \end{align*}
 where we used \eqref{eq:gen_triangle_inequality} and Assumption \ref{ass:p-R_regularity_condition_on_noise} for the first and second inequality respectively.  
 
As noted in the introduction, the focus of this paper is on the convergence rate of the error $e_{k}\defeq u_{k} - U_{k}$ and not on, say, the covariance operator of $e_{k}$, though that information is also important in applications. Note that if $\xi_k(\tstep)$ in Assumption \ref{ass:p-R_regularity_condition_on_noise} does not belong to $L^2_{\mathbb{P}}$, then $\xi_k(\tstep)$ does not admit a covariance operator.
Accordingly, Assumption \ref{ass:p-R_regularity_condition_on_noise} and similar assumptions later in the paper are only upper bounds, and we do not actually work with the covariance operator of $\xi_{k}$.
The precise construction of stochastic models for discretisation and truncation error is an interesting topic in its own right at the interface of numerical analysis and probability, upon which this paper only starts to touch;
we anticipate that there will be further research concerning this question.

Given $e_{k} = u_{k} - U_{k}$, it follows from \eqref{eq:integrand_g} and \eqref{eq:randomised_numerical_scheme_0} that
\begin{equation}
	\label{eq:strong_error_e}
	e_{k + 1} = \bigl( \odeflow^{\tstep}(u_{k}) - \odeflow^{\tstep}(U_{k}) \bigr) - \bigl( \numflow^{\tstep}(U_{k}) - \odeflow^{\tstep}(U_{k}) \bigr) - \xi_{k}(\tstep).
\end{equation}
We shall use the decomposition \eqref{eq:strong_error_e} throughout this article.
	
The next result is stronger than \citet[Theorem~2.2]{Conrad:2016}, as the discrete time supremum is inside the expectation, and as it does not require the vector field $f$ to be globally Lipschitz nor $\xi$ to be Gaussian:

\begin{theorem}
	\label{thm:strong_error_sup_inside}
	Suppose Assumptions \ref{ass:exact_flow} and \ref{ass:numerical_flow} hold, and fix $u_0=U_0$.
	Furthermore, if it holds that $X\in L^2_{\mathbb{P}} \implies \numflow^\tstep (X)\in L^2_{\mathbb{P}}$, and if the $(\xi_k(\tstep))_{k\in[K]}$ have zero mean, are mutually independent, and satisfy Assumption \ref{ass:p-R_regularity_condition_on_noise} for $R=2$ and $p\geq 1$, then there exists $C>0$ that does not depend on $\tstep$ such that
	\begin{equation}
		\label{eq:strong_error_sup_inside_discrete}
		\Expect \biggl[ \max_{ k \in [K]} \norm{ e_{k} }^{2} \biggr] \leq C \tstep^{2 p \wedge 2 q}.
	\end{equation}
\end{theorem}

In contrast to Theorem \ref{thm:strong_error_sup_inside}, which required that the $(\xi_k(\tstep))_{k\in[K]}$ be independent and centred in order to construct a martingale, we make no independence or centredness assumptions on the $(\xi_k(\tstep))_{k\in[K]}$ for the rest of this article.
The following result should be compared to Theorem \ref{thm:strong_error_sup_inside} by considering the case $R=n=2$.
Then for the randomised method to have the same order as the deterministic method on which it is based, we need that $p\geq q + \tfrac{1}{2}$.
In other words, if we remove the assumptions on the $(\xi_{k}(\tstep))_{k\in[K]}$ of independence and centredness, then we require that the second moments of the $(\xi_{k}(\tstep))_{k\in[K]}$ decay to zero with time-step $\tau$ at a faster rate than in Theorem \ref{thm:strong_error_sup_inside}, since the lower bound $q + \tfrac{1}{2}$ on $p$ implied by Theorem \ref{thm:convergence_without_iid_mean_zero_noise} is larger than the lower bound $q$ on $p$ implied by Theorem \ref{thm:strong_error_sup_inside}.

\begin{theorem}
	\label{thm:convergence_without_iid_mean_zero_noise}
	Let $n \in \Naturals$.
	Suppose that Assumptions \ref{ass:exact_flow}, \ref{ass:numerical_flow}, and \ref{ass:p-R_regularity_condition_on_noise} hold with $\tstep^\ast\leq 1$, $q\geq 1$, $p\geq 1$, and $R$, and that $u_0=U_0$.
	Then, there exists a $\overline{C}>0$ that does not depend on $\tstep$ such that for $0<\tstep<\tstep^\ast$,
	\begin{equation}
		\label{eq:convergence_without_iid_mean_zero_noise}
		\Expect\left[\max_{\ell\in[K]}\norm{e_\ell}^n\right] \leq \overline{C}\tstep^{n(q\wedge (p-1/2))}.
	\end{equation}
	where
	\begin{equation}
		\label{eq:C_overline}
		\overline{C}\defeq 2 T \max\{(4C_\numflow)^n,(2 C_{\xi,R})^n\}\exp\left(TC_\odeflow(n,\tstep^\ast)\right)
	\end{equation}
	and	$C_\odeflow(n,\tstep^\ast)$ is defined according to \eqref{eq:C_odeflow_n}.
\end{theorem}


We shall show that if we strengthen Assumption \ref{ass:p-R_regularity_condition_on_noise} by allowing for arbitrarily large $R\in\Naturals$, then the moment generating function of $\max_{\ell\in [K]}\norm{e_\ell}^n$ is finite on $\Reals$.

\begin{corollary}
	\label{cor:finite_mgf_lipschitz_flow}
	Fix $n \in \Naturals$.
	Suppose that Assumptions \ref{ass:exact_flow} and \ref{ass:numerical_flow} hold, and that Assumption \ref{ass:p-R_regularity_condition_on_noise} holds with $R= + \infty$ and $p\geq 1/2$.
	Then, for all $0<\tstep<\tstep^\ast$ and all $\rho \in \Reals$,
	\begin{equation}
		\label{eq:finite_mgf_error}
		\Expect\left[\exp\left(\rho \max_{\ell\in[K]}\norm{e_\ell}^{n}\right)\right] < \infty .
	\end{equation}
\end{corollary}
Hence, by Markov's inequality, the distribution of \newline $\max_{\ell\in [K]}\norm{e_{\ell}}^n$ concentrates exponentially about its \newline mean.

We close this section by noting that, while we have made no attempt to find the optimal constants in Theorem \ref{thm:strong_error_sup_inside} and Theorem \ref{thm:convergence_without_iid_mean_zero_noise}, the convergence orders in these results cannot be improved at the present level of generality. This is because the convergence order of the randomised solution cannot exceed that of the underlying deterministic solver, unless the random variables $\xi_k(\tstep)$ used to model the error $\odeflow^{\tstep}(u_k)-\numflow^{\tstep}(u_k)$ at each time step $t_k$ are chosen to achieve this effect. We leave the construction of such randomised solvers for future work.

\section{Integration for locally Lipschitz vector fields}
\label{sec:dissipative}

This section considers the numerical integration of vector fields $f$ that satisfy the following polynomial growth condition.

\begin{assumption}
	\label{ass:polynomial_growth}
	The vector field $f$ is continuously differentiable, and both $f$ and the associated map $\odeflow^\tstep$ defined by \eqref{eq:odeflow} admit $0<\tstep^\ast\leq 1$, $C_\odeflow \geq 1$, and $s \geq 1$, such that the following inequalities hold for all $a,b\in\Reals^{d}$ and all $0<\tstep<\tstep^\ast$:
\begin{subequations}
	\begin{align}
		\norm{ f(a)-f(b) }& \leq C_\odeflow ( 1 + \norm{a}^s + \norm{b}^s ) \norm{a-b}
		\label{eq:polynomial_growth_of_vector_field}
		\\
		\norm{ \odeflow^\tstep(a)-\odeflow^\tstep(b) }& \leq\left( 1 + \tstep C_\odeflow \left(1 + \norm{a}^s + \norm{b}^s \right)\right) \norm{a-b}.
		\label{eq:polynomial_growth_of_flow_map}
	\end{align}
\end{subequations}

\end{assumption}
The inequality \eqref{eq:polynomial_growth_of_vector_field} implies
	\begin{align}
		\norm{f(a)}&\leq \norm{f(a)-f(0)} + \norm{f(0)}
		\notag\\
		&\leq C_\odeflow(1 + \norm{a}^s)\norm{a} + \norm{f(0)}
		\label{eq:bound_on_norm_f_of_a}.
	\end{align}	
By Taylor's theorem, the remainder term $R^\tstep(a)$ in the first-order Taylor expansion \eqref{eq:taylor_remainder_term} of $\odeflow^\tstep(a)$ is given by the derivatives of $f$, evaluated at some $a\in\Reals^{d}$ for some $0\leq t\leq \tstep$.
The condition \eqref{eq:polynomial_growth_of_flow_map} means that for some $\tstep^\ast>0$ that is sufficiently small, the norm of the difference between two remainder terms can be controlled.
The growth condition \eqref{eq:polynomial_growth_of_vector_field} is not new; see for example \citet[Assumption~4.1]{HighamStuartMao:2002}.

The following result is analogous to Theorem \ref{thm:convergence_without_iid_mean_zero_noise}.
It states that we can replace Assumption \ref{ass:exact_flow} with Assumption \ref{ass:polynomial_growth} and obtain the same result as Theorem \ref{thm:convergence_without_iid_mean_zero_noise}, provided that the $(\xi_{k}(\tstep))_{k\in[K]}$ are $\mathbb{P}$-a.s.\ bounded.

\begin{theorem}	\label{thm:strong_error_sup_inside_polynomial_growth_and_bounded_noise}
	Suppose that Assumptions \ref{ass:polynomial_growth}, \ref{ass:numerical_flow}, and \ref{ass:p-R_regularity_condition_on_noise} hold for $p$ and $R$ as in Theorem \ref{thm:convergence_without_iid_mean_zero_noise}.
	Suppose that $u_0=U_0$.
	If the $(\xi_{k}(\tstep))_{k\in[K]}$ are $\mathbb{P}$-a.s. uniformly bounded over all $k$ by a positive scalar that is $O(\tstep)$, then the conclusions of Theorem \ref{thm:convergence_without_iid_mean_zero_noise} hold.
\end{theorem}

It is of theoretical interest to determine whether there exists a deterministic numerical method $\numflow$ such that the randomised version given by \eqref{eq:randomised_numerical_scheme_0} has the same order even when each $\xi_{k}(\tstep)$ is not $\mathbb{P}$-a.s.\ bounded.
In the remainder of this section, we shall show that for the implicit Euler method $\numflow^\tstep \colon \Reals^{d}\to\Reals^{d}$ defined by
\begin{equation}
	\label{eq:implicit_euler}
	\numflow^{\tstep}(a)\defeq a + \tstep f(\numflow^{\tstep}(a)),
\end{equation}
the randomised version given by \eqref{eq:randomised_numerical_scheme_0} has order 1, under the following dissipativity assumption.

\begin{assumption}
	\label{ass:dissipative_vector_field}
	The function $f$ admits parameters $\alpha \geq 0$ and $\beta \in \Reals$ such that
	\begin{equation}
		\label{eq:dissipativity_property}
		\innerprod{ f(v) }{ v } \leq \alpha + \beta \norm{v}^2 \text{ for all $v \in \Reals^{d}$.}
	\end{equation}
\end{assumption}

Assumption \ref{ass:dissipative_vector_field} is more general than the usual dissipativity property found in \citet[Equation (1.2)]{HumphriesStuart:1994} because $\beta$ may assume positive values.
The sign of $\beta$ in \eqref{eq:dissipativity_property} plays an important role in the behaviour of the solution $u$ of \eqref{eq:the_ivp}, as well as in numerical methods for solving for $u$.
For example, if $\beta$ is positive, then the problem \eqref{eq:the_ivp} may be stiff.
In this paper, we study only the rate of convergence, and leave the issue of stiffness for future work.
In particular, allowing for positive $\beta$ poses no problem for establishing moment bounds, as we show in Lemma \ref{lem:as_bound_on_abs_Un_sq}.

Recent studies in numerical methods for stochastic differential equations consider constraints on the drift that feature the same right-hand side as \eqref{eq:dissipativity_property}, e.g.\ \citet{Fang:2016} and \citet{Mao:2013}.
We reiterate, however, that the analysis of numerical methods for stochastic differential equations cannot be applied to probabilistic solvers of the form \eqref{eq:randomised_numerical_scheme_0}, because of the different behaviour in the additive noise (see e.g. Assumption \ref{ass:p-R_regularity_condition_on_noise}).


\begin{assumption}
	\label{ass:tau}
	Let $\tstep^\ast\leq 1$ be as in Assumption \ref{ass:polynomial_growth} and $\beta\in\Reals$ be as in Assumption \ref{ass:dissipative_vector_field}.
	Then there exists some $0<\tstep'\leq \min\{\tstep^\ast,(2\absval{ \beta })^{-1}\}$ such that there exists a solution $\numflow^\tstep(a)$ to the implicit equation \eqref{eq:implicit_euler} for every $0\leq \tstep\leq \tstep'$, such that the solution $\numflow^\tstep(a)$ varies continuously as a function of $\tstep$ in the interval $0\leq \tstep\leq\tstep'$, and such that $\left.\numflow^\tstep\right\vert_{\tstep=0}(a) = a$.
\end{assumption}

Note that Assumption \ref{ass:tau} is weaker than assuming unique solvability of \eqref{eq:implicit_euler} for every $a\in\Reals^{d}$ over a sufficiently small time interval.

Unless otherwise specified, we shall assume hereafter that $0<\tstep<\tstep'$.

\subsection{Moment bounds for implicit Euler}
\label{subsec:moment_bounds}

\begin{lemma}
	\label{lem:as_bound_on_abs_Un_sq}
	Suppose that Assumptions \ref{ass:polynomial_growth}, \ref{ass:dissipative_vector_field}, and \ref{ass:tau} hold, and let $n\in\Naturals$ be arbitrary.
	Given a fixed, deterministic $U_0$, the following holds uniformly in $\omega\in \ProbSpace$:
	\begin{equation}
		\label{eq:abs_Un}
		\max_{i\in [T/\tstep]} \norm{U_i}^{2n} \leq (2C_2)^n \left[ 1 + \tstep^{-n} \left( \sum_{i=1}^{T/\tstep} \norm{\xi_i(\tstep)}^2 \right)^{n}\, \right],
	\end{equation}
	for $C_2$ given in \eqref{eq:C_2} below.
\end{lemma}
Note that Lemma \ref{lem:as_bound_on_abs_Un_sq} is the only statement for which we directly use Assumption \ref{ass:dissipative_vector_field}.
The following results depend on Assumption \ref{ass:dissipative_vector_field} only insofar as they depend on the conclusions of Lemma \ref{lem:as_bound_on_abs_Un_sq}.

\begin{proposition}
	\label{prop:bounds_moments_max_num_soln}
	Suppose that Assumptions \ref{ass:polynomial_growth}, \ref{ass:tau}, and \ref{ass:dissipative_vector_field} hold, and let $n\in\Naturals$ be arbitrary.
	If Assumption \ref{ass:p-R_regularity_condition_on_noise} holds for some $R \geq 2n$ and some $p \geq 1$, then
	\begin{equation}
		\label{eq:Eval_max_i_abs_Ui}
		\Expect \biggl[ \max_{i\in[K]} \norm{U_i}^{2n} \biggr] \leq (2C_2)^n\left(1 + \left(TC_{\xi,R}^2\tstep^{ 2p-1}\right)^n\right),
	\end{equation}
	for $C_2$ defined in \eqref{eq:C_2}, and $C_{\xi,R}$ in Assumption \ref{ass:p-R_regularity_condition_on_noise}.
\end{proposition}

\begin{proof}
The statement follows directly from the conclusion \eqref{eq:abs_Un} of Lemma \ref{lem:as_bound_on_abs_Un_sq} and \eqref{eq:upper_bound_expectation_nth_power_sum_norm_xi_sq} with $w=2$ and $v=n$.
\end{proof}

\begin{corollary}
	\label{cor:finite_mgf}
	Suppose that Assumptions \ref{ass:polynomial_growth}, \ref{ass:tau}, \ref{ass:dissipative_vector_field}, and \ref{ass:p-R_regularity_condition_on_noise} hold with $R= + \infty$ and $p\geq 1/2$.
	Then 
	\begin{equation*}
	 \Expect\left[\exp\left(\rho\max_{i\in [K]} \norm{ U_i }^{2}\right)\right]<\infty,\quad\text{for all }\rho\in\Reals.
	\end{equation*}
\end{corollary}

\begin{proof}
	The result follows from Proposition \ref{prop:bounds_moments_max_num_soln}, the series expansion of the exponential, and the dominated convergence theorem.
\end{proof}

Lemma \ref{lem:as_bound_on_abs_Un_sq} shows that whenever Assumption \ref{ass:dissipative_vector_field} holds, then regardless of the growth behaviour of $f$, the randomised implicit Euler method has the property that if $X\in L^{R}_{\mathbb{P}}$ for some $R\in\Naturals$, then $\numflow^\tstep(X)\in L^{R}_{\mathbb{P}}$ as well;
cf.\ the hypothesis on $\numflow^\tstep$ in Theorem \ref{thm:strong_error_sup_inside}.

\subsection{Convergence in discrete time for implicit Euler}
\label{subsec:discrete_time_convergence_implicit_euler}

\begin{proposition}
\label{prop:W_k}
	Let $n\in\Naturals$, and suppose that Assumptions \ref{ass:polynomial_growth} and \ref{ass:p-R_regularity_condition_on_noise} hold for some $R \geq 2n(2s + 1)$ and some $p \geq 1$.
	Then there exists a scalar $C_\numflow>0$ that does not depend on $\tstep$ or $k\in[K]$, such that for all $k\in[K]$,
	\begin{equation}
		\label{eq:W_k_estimate}
		\Expect \bigl[ \Norm{\numflow^{\tstep}(U_{k})-\odeflow^{\tstep}(U_{k})}^{2n} \bigr] \leq C_\numflow\tstep^{4n}
	\end{equation}
	with $C_\numflow$ as in \eqref{eq:C_numflow}.
\end{proposition}

Proposition \ref{prop:W_k} shows that when $f$ satisfies the polynomial growth condition and $\numflow$ is the implicit Euler method, then the local truncation error at step $k$ of the randomised numerical integrator satisfies a bound analogous to that in Assumption \ref{ass:numerical_flow}, provided that the random variables $(\xi_{k}(\tstep))_{k\in[K]}$ are sufficiently regular.

\begin{theorem}	\label{thm:convergence_without_iid_mean_zero_noise_implicit_euler}
	Let $n\in\Naturals$, and let $\numflow^\tstep$ be given by \eqref{eq:implicit_euler}.
	Suppose that Assumptions \ref{ass:polynomial_growth}, \ref{ass:dissipative_vector_field}, and \ref{ass:tau} hold, with parameters $s\geq 1$ and $\tstep'>0$.
	Suppose that Assumption \ref{ass:p-R_regularity_condition_on_noise} holds with $R\geq 2n(2s + 1)$ and $p\geq \tfrac{3}{2}$.
	Then there exists some $C>0$ that does not depend on $\tstep$ such that for $0<\tstep<\tstep'$,
	\begin{equation*}
		\Expect\left[\max_{\ell\in [K]}\norm{e_{\ell}}^{2n}\right]\leq C\tstep^{2n}.
	\end{equation*}
\end{theorem}
Note that the condition $p\geq \tfrac{3}{2}$ is the same condition $p\geq q + \tfrac{1}{2}$ on $p$ in Theorem \ref{thm:convergence_without_iid_mean_zero_noise}, since the implicit Euler method has order $q=1$.

\subsection{Alternative decomposition of the error}

The decomposition \eqref{eq:strong_error_e} of the error $e_{k + 1}$ was used to derive the convergence results above.
One might consider instead using the decomposition
\begin{equation*}
	e_{k + 1} = (\Phi^\tstep(u_k)-\numflow^{\tstep}(u_k)) + (\numflow^{\tstep}(u_k)-\numflow^{\tstep}(U_k))-\xi_k(\tstep)
\end{equation*}
with the goal of using some stability properties of the implicit Euler method.
However, this approach leads to a convergence result that is weaker, either because it requires exponential integrability of $\norm{U_{k}}$, or because the convergence is uniform only on a proper subset $\ProbSpace_{\tstep}$ of the event space $\ProbSpace$.
Recall that we do not assume any of the $\xi_{k}(\tstep)$ to be a.s.\ bounded.

By \eqref{eq:conseq_of_Young} and the fact that implicit Euler has order one (i.e. Assumption \ref{ass:numerical_flow})
\begin{align}
	\norm{ e_{k + 1} }^2
	& \leq \left(\norm{ \Phi^\tstep (u_k)-\numflow^{\tstep}(u_k)}\right. 
	\notag
	\\
	&\quad + \left.\norm{\numflow^{\tstep}(u_k)-\numflow^{\tstep}(U_k)-\xi_k(\tstep)}\right)^2
	\notag\\
	& \leq (1 + \tstep^{-1}) \norm{ \Phi^\tstep (u_k)-\numflow^{\tstep}(u_k) }^2 
	\notag\\
	&\quad+ (1 + \tstep) \norm{ \numflow^{\tstep}(u_k)-\numflow^{\tstep}(U_k)-\xi_k(\tstep) }^2
	\notag\\
	&\leq (1 + \tstep^{-1})(C\tstep^{2})^2 \label{eq:prelim01}\\
	&\quad + (1 + \tstep) \norm{ \numflow^{\tstep}(u_k)-\numflow^{\tstep}(U_k)-\xi_k(\tstep) }^2,
	\notag
\end{align}
where one can show, using the proof of Proposition \ref{prop:W_k}, that $C>0$ in \eqref{eq:prelim01} depends on $\norm{u_k}^s$ but not on $\tstep$.
By \eqref{eq:conseq_of_Young} we obtain
\begin{align*}
	&\norm{ \numflow^{\tstep}(u_k)-\numflow^{\tstep}(U_k)-\xi_k(\tstep) }^2 
	\\
	&\leq (1 + \tstep) \norm{ \numflow^{\tstep}(u_k)-\numflow^{\tstep}(U_k) }^2 + (1 + \tstep^{-1}) \norm{ \xi_k(\tstep) }^2.
\end{align*}
Substituting the result above into \eqref{eq:prelim01}, and assuming that $\tstep<1$, we obtain
\begin{align}
	\label{eq:prelim03}
	\norm{ e_{k + 1} }^2 \leq& C\tstep^{3} + (1 + \tstep)^2 \norm{ \numflow^{\tstep}(u_k) - \numflow^{\tstep}(U_k) }^2 
	\\
	&\quad+ 4\tstep^{-1} \norm{ \xi_{k}(\tstep) }^2.
	\notag
\end{align}
The definition \eqref{eq:implicit_euler} of the implicit Euler method and \eqref{eq:conseq_of_Young} yield
\begin{align*}
	&\norm{\numflow^{\tstep}(u_k)-\numflow^{\tstep}(U_k)}^2
	\\
	& = \norm{ u_k-U_k + \tstep f(\numflow^{\tstep}(u_k))-\tstep f(\numflow^{\tstep}(U_k)) }^2
	\\
	&\leq (1 + \tstep)\norm{ u_k-U_k}^2 
	\\
	&\quad +(1 + \tstep^{-1})\tstep^2\norm{ f(\numflow^{\tstep}(u_k))-f(\numflow^{\tstep}(U_k))}^2
	\\
	& \leq (1 + \tstep)\norm{ e_k}^2 + (1 + \tstep)\tstep D^2 \norm{\numflow^{\tstep}(u_k)-\numflow^{\tstep}(U_k)}^2\times
	\\
	&\quad \left[1 + \norm{\numflow^{\tstep}(u_{k})}^s + \norm{\numflow^{\tstep}(U_{k})}^s\right]^2,
\end{align*}
by Assumption \ref{ass:polynomial_growth}.
Rearranging the above yields
\begin{align*}
	(1 + \tstep)\norm{e_k}^2\geq& \norm{\numflow^\tstep(u_k)-\numflow^\tstep(U_k)}^2 (1-(1+\tstep)\tstep D^2\hat{M})
\end{align*}
where $\hat{M}\defeq [1 + \norm{\numflow^\tstep(u_k)}^s + \norm{\numflow^\tstep(U_k)}^s]^2$ is a random variable. Analogously, define the random variable $M$ by
\begin{equation*}
M\defeq \left[1 + \max_{k\in [K]}\norm{\numflow^\tstep(u_k)}^s + \max_{k\in [K]}\norm{\numflow^\tstep(U_k)}^s\right]^2.
\end{equation*}
Suppose that $u_0=U_0$ are fixed, and define
\begin{equation}
	\label{eq:omega_tau}
	\ProbSpace_\tstep \defeq \left\{ \omega\in\ProbSpace \,\middle|\, 1-(1 + \tstep)\tstep D^2 M(\omega) >0\right\}.
\end{equation}
Since it is not the case that all of the random variables $(\xi_{k}(\tstep))_{k\in[K]}$ are a.s.-bounded, it follows that $\ProbSpace_\tstep$ is a proper subset of $\ProbSpace$, for every $\tstep>0$.
In what follows, we assume that $\ProbSpace_\tstep$ is nonempty, and that $\omega\in\ProbSpace_\tstep$;
we suppress the $\omega$--dependence of all random variables.
Define $\widetilde{C}>0$ by
\begin{align}
	\left(1-(1 + \tstep)\tstep D^2M\right)^{-1} & =\sum_{n=0}^\infty \left[(1 + \tstep)\tstep D^2 M\right]^n \notag
	\\
	&\qefed 1 + \widetilde{C}\tstep,
	\label{eq:C_tilde}
\end{align}
Using \eqref{eq:C_tilde}, we have
\begin{equation*}
	 \norm{ \numflow^{\tstep}(u_k) - \numflow^{\tstep}(U_k) }^2 \leq (1 + \tstep)(1 + \widetilde{C}\tstep)\norm{e_{k}}^2,
\end{equation*}
and substituting the above into \eqref{eq:prelim03} yields
\begin{equation*}
	\norm{ e_{k + 1}}^2\leq C\tstep^{3} + (1 + \tstep)^3(1 + \widetilde{C}\tstep)\norm{e_{k}}^2 + 4\tstep^{-1}\norm{\xi_{k}(\tstep)}^2.
\end{equation*}
Proceeding as in the proof of Theorem \ref{thm:convergence_without_iid_mean_zero_noise_implicit_euler}, we use a telescoping sum, Gr{\"o}nwall's theorem, and Assumption \ref{ass:p-R_regularity_condition_on_noise} with $p\geq q + 1/2$ to obtain
\begin{equation}
	\label{eq:inferior_convergence_result}
	\Expect\left[1_{\ProbSpace_\tstep}\max_{k\in[K]}\norm{e_{k}}^2\right] \leq \Expect\left[1_{\ProbSpace_\tstep}\exp\left(T\kappa\right)\right]C\tstep^{2},
\end{equation}
where $\kappa$ depends on $\tstep$ according to
\begin{align*}
	\kappa(\tstep) &\defeq \tstep^{-1}\left[(1 + \tstep)^3(1 + \widetilde{C}\tstep)-1\right]
	\\
	&=\widetilde{C}\tstep^3 + (3\widetilde{C} + 1)\tstep^2 + 3(\widetilde{C} + 1)\tstep + (3 + \widetilde{C}).
\end{align*}
For any $\tstep>0$, it follows from the definition of $\kappa$, and considering the zeroth order term $3 + \widetilde{C}$ above that
\begin{equation*}
	\kappa(\tstep) > 3 + (1 + \tstep) D^2 M.
\end{equation*}
From \eqref{eq:omega_tau}, it follows that, for all $\omega\in \ProbSpace_\tstep$, we have
\begin{equation*}
  (1+\tstep)D^2 M< \tstep^{-1}.
\end{equation*}
where the right-hand side increases to infinity as $\tstep$ decreases to zero. Thus, it need not be true that the quantity $\Expect[1_{\ProbSpace_\tstep}\exp(T\kappa)]$ is finite.
One way to ensure that $\Expect[1_{\ProbSpace_\tstep}\exp(T\kappa)]$ is finite for $0<\tstep<\tstep'$ would be to require that $\Expect[\exp(T\kappa)]$ is finite on the same range.
By the inequality for $\kappa$ above, a necessary condition for $\Expect[\exp(T\kappa)]$ to be finite is exponential integrability of $\max_{k\in[K]}\norm{\numflow^{\tstep}(U_{k})}^{2s}$. In many cases, a necessary condition for this would be exponential integrability of $\max_{k\in[K]}\norm{U_{k}}^{2s}$.
By Corollary \ref{cor:finite_mgf}, in order to guarantee exponential integrability of $\max_{k\in[K]}\norm{U_{k}}^{2s}$, we would need to impose much stronger regularity conditions on the $(\xi_{k}(\tstep))_{k\in[K]}$ than those in Theorem \ref{thm:convergence_without_iid_mean_zero_noise_implicit_euler}.
Finally, we also remark that if the $(\xi_k)_{k\in [K]}$ are not $\mathbb{P}$-a.s. uniformly bounded, then for any $\tstep>0$, \eqref{eq:inferior_convergence_result} is a weaker convergence result than Theorem \ref{thm:convergence_without_iid_mean_zero_noise_implicit_euler}, since in this case for any $\tstep>0$ $\ProbSpace_\tstep$ will be a proper subset of $\ProbSpace$.

\section{Additional results}
\label{sec:additional_results}

\subsection{Convergence for continuous-time interpolant}

Recall \eqref{eq:randomised_numerical_scheme_0} defines the discrete-time process $(U_{k})_{k\in[K]}$;
in many applications, it is often useful to have a numerical method that provides continuous output, e.g.\ an inverse problem or data assimilation that requires comparison between the numerical solution and an observation that is not on the time grid $(t_k)_{k\in[K]}$ defined in \eqref{eq:time_grid}.
Given this time grid $(t_k)_{k\in[K]}$, we may define a continuous-time process $U$ by
\[
	U(t) \defeq \numflow^{t - t_{k}}(U_{k}) + \xi_{k}(t - t_{k})
	\quad
	\text{for $t \in [t_{k}, t_{k + 1})$.}
\]
For the above definition to work, we assume that each $\xi_{k}$ is a stochastic process defined on the time interval $[0,\tstep]$. In addition, to ensure that the process $U$ has $\mathbb{P}$-almost surely continuous paths, we require that $\mathbb{P}(\xi_k(0)=0)=1$.
The corresponding notion of the error at time $0\leq t\leq T$ is given by $e(t)\defeq u(t)-U(t)$, where $u(t)=\odeflow^{t}(u_0)$.
We emphasise that the continuous-time process $(U(t))_{0\leq t\leq T}$ described above will in general differ from the continuous-time process obtained by linear interpolation of $(U_{k})_{k\in[K]}$.

We now demonstrate how one can obtain a convergence result for the continuous-time process from a discrete-time convergence result by strengthening the assumption on the noise, using Theorem \ref{thm:convergence_without_iid_mean_zero_noise} as an example.
Consider the following version of Assumption~\ref{ass:p-R_regularity_condition_on_noise}:

\begin{assumption}
	\label{ass:p-R_regularity_condition_on_noise_continuous_time}
	Fix $\tstep>0$.
	The collection $(\xi_k)_{k\in\Naturals}$ of stochastic processes $\xi_k \colon \Omega\times [0,\tstep]\to\Reals^{d}$ satisfies $\mathbb{P}(\xi_k(0)=0)=1$ and admits $p \geq 1$, $R\in\Naturals\cup\{ + \infty\}$ and some $C_{\xi,R}\geq 1$ that do not depend on $k\in\Naturals$ or $\tstep$, such that for all $1\leq r\leq R$ and for all $k\in\Naturals$,
	\begin{equation*}
		\Expect \biggl[ \sup_{0<t\leq \tstep}\norm{\xi_{k}(t)}^{r} \biggr] \leq\left( C_{\xi,R}\tstep^{p + 1/2}\right)^r.
	\end{equation*}	
\end{assumption}

Recall that we do not assume that the $\xi_{k}$ are independent, identically distributed, or centred.

\begin{theorem}
	\label{thm:strong_error_sup_inside_continuous}
	Let $n\in\Naturals$, and suppose that Assumptions \ref{ass:exact_flow}, \ref{ass:numerical_flow}, and \ref{ass:p-R_regularity_condition_on_noise_continuous_time} hold with parameters $\tstep^\ast$, $C_\odeflow$, $C_\numflow$, $q$, $C_{\xi,R}$, $p$, and $R$.
	Then for all $0<\tstep<\tstep^\ast$,
	\begin{align}
		\Expect&\left[\sup_{0\leq t\leq T}\norm{e(t)}^n\right]
		\notag\\
		&\leq 3^{n-1}\left(\left(1 + C_\odeflow \tstep^\ast\right)^n\overline{C} + C_\numflow^n(\tstep^\ast)^n + TC^n_{\xi,R}\right)\times
		\label{eq:strong_error_sup_inside_continuous}
		\\
		&\quad\tstep^{n(q\wedge (p-1/2))},
		\notag
	\end{align}
	where $\overline{C}$ is defined in \eqref{eq:C_overline}.
\end{theorem}

The next result follows from Theorem \ref{thm:strong_error_sup_inside_continuous} in the same way that Corollary \ref{cor:finite_mgf_lipschitz_flow} follows from Theorem \ref{thm:convergence_without_iid_mean_zero_noise}.

\begin{corollary}
	Fix $n\in\Naturals$.
	Suppose that Assumptions \ref{ass:exact_flow} and \ref{ass:numerical_flow} hold, and that Assumption \ref{ass:p-R_regularity_condition_on_noise_continuous_time} holds with $R= + \infty$ and $p\geq 1/2$.
	Then, for all $0<\tstep<\tstep^\ast$,
	\begin{equation}
		\Expect\left[\exp\left(\rho\sup_{0\leq t\leq T}\norm{e(t)}^n\right)\right]<\infty,\quad\text{for all }\rho\in\Reals.
	\end{equation}
\end{corollary}

\begin{proof}
 	The proof follows by the series representation of the exponential and the dominated convergence theorem;
 	see the proof of Corollary \ref{cor:finite_mgf_lipschitz_flow}.
\end{proof}

\subsection{Existence of processes that satisfy the $(p,R)$-regularity condition}

The lemma below shows that there exist random variables that are not $\mathbb{P}$-a.s.\ bounded, and that satisfy Assumption \ref{ass:p-R_regularity_condition_on_noise} and, more generally, Assumption \ref{ass:p-R_regularity_condition_on_noise_continuous_time} for $R= + \infty$.

\begin{lemma}
	\label{lem:rth_power_of_scaled_integrated_BM}
	Let $\tstep>0$ and $p \geq 1$ be arbitrary, and let $(B_t)_{0 \leq t \leq \tstep}$ be $\Reals^{d}$-valued Brownian motion.
	Then 
	\begin{equation*}
	 \xi_0(t) \defeq \tstep^{p-1}\int_0^t B_s \, \rd s
	\end{equation*}
	satisfies
	\begin{equation}
		 \label{eq:scaling_relation_noise_process}
		 \Expect \biggl[ \sup_{t \leq \tstep}\norm{\xi_0(t)}^r \biggr] \leq 4
	\tstep^{r(p + 1/2)} \text{ for all $r \in \Naturals$.}
	\end{equation}
\end{lemma}

Note that variants of the integrated Brownian motion process have been used for modelling local truncation error in other works \citep{Schober:2014,Conrad:2016}.
However, the point of Lemma~\ref{lem:rth_power_of_scaled_integrated_BM} is \textit{not} to suggest that the local truncation error behaves as an integrated Brownian motion, nor even that the integrated Brownian motion process is a suitable model for the local truncation error.
The point of Lemma \ref{lem:rth_power_of_scaled_integrated_BM} is simply to show that there exist processes that satisfy Assumption \ref{ass:p-R_regularity_condition_on_noise_continuous_time} with $R= + \infty$.
The construction of models that better reflect known properties of the truncation error, for specific classes of vector fields $f$, is an interesting task that we leave for future work.

\bibliographystyle{abbrvnat}
\bibliography{references}   

\appendix
\section{Proofs}
\label{sec:appendix}

\begin{proof}[Proof of Lemma \ref{lem:gen_peter_paul_inequality}]
	The assertion \eqref{eq:gen_peter_paul_inequality} holds immediately for $n=1$, so let $n\in\Naturals\setminus\{1\}$, and recall the binomial formula: for $x,y\in\Reals$ and $n\in \Naturals\setminus\{1\}$,
	\begin{equation*}
		(x + y)^n = \sum^n_{k=0} \binom{ n }{ k } x^{k} y^{n - k} = x^n + y^n + \sum^{n-1}_{k=1} \binom{ n }{ k } x^{k} y^{n-k}.
	\end{equation*}
	Fix $\delta>0$.
	By \eqref{eq:Young}, for any $1\leq k\leq n-1$,
	\begin{align*}
		x^{k} y^{n-k} &\leq \delta \frac{k}{n}x^{n} + \frac{1}{\delta^{k/(n-k)}}\frac{n-k}{n}y^n
		\\
		&\leq \delta \frac{k}{n}x^{n} + \frac{1}{\delta^{n-1}}\frac{n-k}{n}y^n,
	\end{align*}
	where the second inequality follows from $-\tfrac{k}{n-k}\geq -(n-1)$.
	Therefore,
	\begin{align}
		(x + y)^n  \leq& x^n \left( 1 + \delta\sum^{n-1}_{k=1}\binom{ n }{ k }\frac{k}{n} \right) 
		\\
		&+ y^n \left( 1 + \frac{1}{\delta^{n-1}}\sum^{n-1}_{k=1}\binom{ n }{ k }\frac{n-k}{n} \right),
		\notag
	\end{align}
	and the proof is complete upon observing that
	\begin{align*}
		\sum^{n-1}_{k=1}\binom{ n }{ k }\frac{k}{n}&=\sum^{n-1}_{k=1} \binom{ n-1 }{ k-1 } = \sum^{n-1}_{j=0} \binom{ n - 1 }{ j } - \binom{ n - 1 }{ n - 1 } \\
		&= (1 + 1)^{n-1} -1 \leq 2^{n-1}
	\end{align*}
	and bounding the other binomial sum in a similar way.\qed
\end{proof}

\begin{proof}[Proof of Theorem \ref{thm:strong_error_sup_inside}]
	By \eqref{eq:strong_error_e},
	\begin{align*}
		\norm{ e_{k + 1} }^{2}
		 =& \Norm{ \bigl( \odeflow^{\tstep}(u_{k}) - \odeflow^{\tstep}(U_{k}) \bigr) - \bigl( \numflow^{\tstep}(U_{k}) - \odeflow^{\tstep}(U_{k}) \bigr) }^{2} 
		\\
		&+ \norm{ \xi_{k}(\tstep) }^{2} + 2 \innerprod{ \odeflow^{\tstep}(u_{k}) - \numflow^{\tstep}(U_{k}) }{ \xi_{k}(\tstep) } .
	\end{align*}
	By \eqref{eq:conseq_of_Young} with $\delta = \tstep$, by Assumption \ref{ass:exact_flow} and Assumption \ref{ass:numerical_flow}, and using that $\tstep<\tstep^\ast\leq 1$,
	\begin{align}
		&\Norm{\odeflow^{\tstep}(u_{k})-\numflow^{\tstep}(U_{k})}^2\notag
		\\
		& = \Norm{ \bigl( \odeflow^{\tstep}(u_{k}) - \odeflow^{\tstep}(U_{k}) \bigr) - \bigl( \numflow^{\tstep}(U_{k}) - \odeflow^{\tstep}(U_{k}) \bigr) }^{2}
		\notag\\
		& \leq (1 + \tstep) \norm{ \odeflow^{\tstep}(u_{k}) - \odeflow^{\tstep}(U_{k}) }^{2} \
		\notag\\
		&\quad + (1 + \tstep^{-1}) \norm{ \numflow^{\tstep}(U_{k}) - \odeflow^{\tstep}(U_{k}) }^{2}
		\notag\\
		& \leq (1 + \tstep) (1 + C_\odeflow \tstep)^{2} \norm{ e_{k} }^{2} + 2C_\numflow^2 \tstep^{1 + 2 q}.
		\label{eq:estimate01}
	\end{align}
	Observe that $[(1 + \tstep)(1 + C_\odeflow\tstep)^2-1]\tstep^{-1}$ equals a quadratic polynomial in $\tstep$ with coefficients $a_0$, $a_1$, and $a_2$.
	Calculating these coefficients and defining
	\begin{equation}
		\label{eq:C_1}
		C_1=C_1(C_\odeflow,\tstep^\ast)\defeq 1 + 2C_\odeflow + C_\odeflow(2 + C_\odeflow)\tstep^\ast + C_\odeflow^2(\tstep^\ast)^2
	\end{equation}
	then yields that $[(1 + \tstep)(1 + C_\odeflow\tstep)^2-1]\tstep^{-1}\leq C_1$ for all $0<\tstep<\tstep^\ast$.
	
	Combining the preceding estimates yields
	\begin{align}
		\label{eq:strong_error_e_3}
		\norm{ e_{k + 1} }^{2} - \norm{ e_{k} }^{2} \leq& C_1 \tstep \norm{ e_{k} }^{2} + 2C_\numflow^2 \tstep^{1 + 2 q} + \norm{ \xi_{k}(\tstep) }^{2} 
		\\
		&+  2 \innerprod{ \odeflow^{\tstep}(u_{k}) - \numflow^{\tstep}(U_{k}) }{ \xi_{k}(\tstep) }.
		\notag
	\end{align}
	Using \eqref{eq:strong_error_e_3} in the telescoping sum
	\[
		\norm{ e_{k} }^{2} - \norm{ e_{0} }^{2} = \sum_{j = 0}^{k - 1} \bigl( \norm{ e_{j + 1} }^{2} - \norm{ e_{j} }^{2} \bigr),
	\]
	the fact that $e_0=u_0-U_0=0$ and $K=T/\tstep$, we obtain
	\begin{align*}
		\norm{e_{k}}^2&\leq \sum^{k-1}_{j=0}\biggr[C_1\tstep \norm{e_j}^2 + C_\numflow\tstep^{1 + 2q} + \norm{\xi_j(\tstep)}^2 
		\notag\\
		&\quad\quad\quad+ 2\Innerprod{\odeflow^\tstep(u_k)-\numflow^\tstep(U_k)}{\xi_k(\tstep)}\biggr]
		\\
		&\leq C_1\tstep\sum^{k-1}_{j=0} \norm{e_j}^2 + \sum^{K-1}_{j=0}\left(2C_\numflow^2\tstep^{1 + 2q} + \norm{\xi_j(\tstep)}^2\right) 
		\\
		&\quad+ 2\Norm{\sum^{k-1}_{j=0}\Innerprod{\odeflow^\tstep(u_k)-\numflow^\tstep(U_k)}{\xi_k(\tstep)}}
		\\
		&\leq C_1\tstep\sum^{k-1}_{j=0} \norm{e_j}^2 + 2TC_\numflow^2\tstep^{2q} + \sum^{K-1}_{j=0}\norm{\xi_j(\tstep)}^2 
		\\
		&\quad+ 2\Norm{\sum^{k-1}_{j=0}\Innerprod{\odeflow^\tstep(u_k)-\numflow^\tstep(U_k)}{\xi_k(\tstep)}}.
	\end{align*}
	It follows from the last inequality that
	\begin{align*}
	 	\max_{\ell\leq k}\norm{e_{\ell}}^2&\leq C_1\tstep\sum^{k-1}_{j=0} \norm{e_j}^2 + 2TC_\numflow^2\tstep^{2q} + \sum^{K-1}_{j=0}\norm{\xi_j(\tstep)}^2
	 	\\
	 	&\quad+ 2\max_{\ell\leq k}\Norm{\sum^{\ell-1}_{j=0}\Innerprod{\odeflow^\tstep(u_k)-\numflow^\tstep(U_k)}{\xi_k(\tstep)}}.
	\end{align*}
	Now replace $\norm{e_j}^2$ on the right-hand side with $\max_{\ell\leq j}\norm{e_\ell}^2$ and take expectations of both sides of the inequality.
	Since Assumption \ref{ass:p-R_regularity_condition_on_noise} holds with $R=2$,
	\begin{equation*}
		\Expect \left[ \sum_{j = 0}^{K - 1}\norm{ \xi_{j}(\tstep) }^{2} \right]\leq \frac{T}{\tstep} (C_{\xi,R}\tstep^{p + 1/2})^2=T C_{\xi,R}^2\tstep^{2p}.
	\end{equation*}
	Next, define for every $k\in[K]$ the $\sigma$-algebra $\mathcal{F}_{j}$ generated by $\xi_{0}(\tstep),\ldots,\xi_{j}(\tstep)$
	Then the sequence $(\mathcal{F}_{j})_{j\in[K]}$ forms a filtration.
	Define $(M_k)_{k\in[K]}$ by
	\begin{equation*}
		M_k\defeq \sum^k_{j=0}\Innerprod{\odeflow^\tstep(u_j)-\numflow^\tstep(U_j)}{\xi_j(\tstep)}.
	\end{equation*}
	We want to show that this process is a martingale with respect to $(\mathcal{F}_{j})_{j\in[K]}$.
	By \eqref{eq:randomised_numerical_scheme_0}, $U_j$ is measurable with respect to $\mathcal{F}_{j-1}$, so $M_k$ is measurable with respect to $\mathcal{F}_k$.
	Hence $(M_k)_{k\in[K]}$ is adapted with respect to $(\mathcal{F}_k)_{k\in[K]}$.
	Observe that
	\begin{align*}
		\Expect\left[\norm{M_k}\right]&\leq \sum^k_{j=0}\Expect\left[ \Norm{\Innerprod{\odeflow^\tstep(u_j)-\numflow^\tstep(U_j)}{\xi_j(\tstep)}}\right]
		\\
		&
		\leq \sum^k_{j=0}\left(\Expect\left[\norm{\odeflow^\tstep(u_j)-\numflow^\tstep(U_j)}^2\right] + \Expect\left[\norm{\xi_j(\tstep)}^2\right]\right)
		\\
		&\leq 2\sum^k_{j=0}\left(\norm{\odeflow^\tstep(u_j)}^2 + \Expect\left[\norm{\numflow^\tstep(U_j)}^2+\norm{\xi_j(\tstep)}^2\right]\right).
	\end{align*}
	Using the assumption that $X\in L^2_{\mathbb{P}} \implies \numflow^\tstep (X)\in L^2_{\mathbb{P}}$, \eqref{eq:randomised_numerical_scheme_0}, Assumption \ref{ass:p-R_regularity_condition_on_noise}, and the fact that $U_0=u_0$ is fixed, it follows that $U_j$ and $\numflow^\tstep(U_j)$ belong to $L^2_{\mathbb{P}}$; thus $M_k$ belongs to $L^1_{\mathbb{P}}$ for every $k\in[K]$.
	We now use the assumption that $\Expect[\xi_j(\tstep)]=0$ for every $j\in[K]$, and that the $(\xi_k(\tstep))_{k\in[K]}$ are mutually independent, in order to establish the martingale property:
	\begin{align*}
		\Expect&\left[M_k-M_{k-1}\middle\vert\mathcal{F}_{k-1}\right] 
		\\
		&= \Expect\left[\Innerprod{\odeflow^\tstep(u_k) - \numflow^\tstep(U_k)}{\xi_k(\tstep)}\middle\vert\mathcal{F}_{k-1}\right],
	\end{align*}
	and the right-hand side vanishes since $U_k$ is measurable with respect to $\mathcal{F}_{k-1}$ as noted earlier.
	Since $(M_k)_{k\in[K]}$ is a martingale, we may apply the Burkholder--Davis--Gundy inequality \citep[Equation (2.2)]{Peskir:1996}.
	Letting $[Y]_{\ell}$ denote the quadratic variation up to time $\ell$ of a process $Y_{k}$, we have
	\begin{align*}
		& \Expect \left[ \max_{k \leq \ell} \Norm{ \sum_{j = 0}^{k - 1} \innerprod{ \odeflow^{\tstep} (u_{j}) - \numflow^{\tstep} (U_{j}) }{ \xi_{j}(\tstep) } } \right]
		\\
		& \quad \leq 3\Expect \Bigl[ [ \innerprod{ \odeflow^{\tstep} (u_{\bullet}) - \numflow^{\tstep} (U_{\bullet}) }{ \xi_{\bullet}(\tstep) } ]_{\ell - 1}^{1/2} \Bigr]\leq 3 \Expect \left[  ab \right]
	\end{align*}
	where we define $b\defeq\sqrt{ \sum_{j = 1}^{\ell - 1} \norm{ \xi_{j}(\tstep) - \xi_{j - 1}(\tstep) }^{2} }$ and $a\defeq\sqrt{\max_{j \leq \ell} \norm{ \odeflow^{\tstep} (u_{j}) - \numflow^{\tstep} (U_{j}) }^2}$. Using \eqref{eq:Young} with the same $a$ and $b$, $r=r^\ast=2$, and $\delta=[6(1 + \tstep)(1 + C_\odeflow\tstep)^2]^{-1}$, and using \eqref{eq:estimate01}, it follows that
	\begin{align*}
		& 3 \Expect \Biggl[\left( \max_{j \leq \ell} \norm{ \odeflow^{\tstep} (u_{j}) - \numflow^{\tstep} (U_{j}) }\right) \sqrt{ \sum_{j = 1}^{\ell - 1} \norm{ \xi_{j}(\tstep) - \xi_{j - 1}(\tstep) }^{2} } \Biggr]
		\\
		& \quad \leq \frac{1}{4}\left(\Expect\left[ \max_{j\leq \ell} \norm{ e_j }^2\right] + 2C_\numflow^2\tau^{1 + 2q}\right)\\
		&\quad \quad+ 9(1 + \tstep^\ast)(1 + C_\odeflow\tstep^\ast)^2\sum^{\ell-1}_{j=1} \Expect\left[\norm{ \xi_{j}(\tstep)-\xi_{j-1}(\tstep) }^2\right]
		\\
		&\quad \leq \frac{1}{4}\left(\Expect\left[ \max_{j\leq \ell} \norm{ e_j }^2\right] + 2C_\numflow^2\tau^{1 + 2q}\right)\\
		&\quad\quad +18(1 + \tstep^\ast)(1 + C_\odeflow\tstep^\ast)^2\sum^{\ell-1}_{j=1} \Expect\left[\norm{ \xi_{j}(\tstep) }^2\right]
	\end{align*}
	where we applied \eqref{eq:conseq_of_Young} with $\delta=1$, $r=r^\ast=2$, $a=\xi_{j}(\tstep)$ and $b=\xi_{j-1}(\tstep)$ to obtain the last inequality.
	Thus by Assumption \ref{ass:p-R_regularity_condition_on_noise} and by using $\ell-1\leq K=T/\tstep$,
	\begin{align*}
		2 \Expect&\left[ \max_{k \leq \ell} \Norm{ \sum_{j = 0}^{\ell - 1} \innerprod{ \odeflow^{\tstep} (u_{j}) - \numflow^{\tstep} (U_{j}) }{ \xi_{j}(\tstep) } } \right]
		\\
		&\leq \frac{1}{2}\left( \Expect \left[ \max_{k\leq \ell} \norm{e_{k}}^2 \right] + 2C_\numflow^2 \tstep^{1 + 2q}\right)
		\\
		&\quad+ 36(1 + \tstep^\ast)(1 + C_\odeflow\tstep^\ast)^2\sum^{\ell-1}_{j=0} \Expect \left[ \norm{\xi_{j}(\tstep)}^2 \right]
		\\
		&\leq \frac{1}{2}\left( \Expect \left[ \max_{k\leq \ell} \norm{e_{k}}^2 \right] + 2C_\numflow^2 \tstep^{1 + 2q}\right) 
		\\
		&\quad+ 36(1 + \tstep^\ast)(1 + C_\odeflow\tstep^\ast)^2T\tstep^{-1} \left(C_{\xi,R}\tstep^{p + 1/2}\right)^2
	\end{align*}
	Combining the preceding estimates, we obtain
	\begin{align*}
		\Expect& \biggl[ \max_{\ell \leq k} \norm{ e_{\ell} }^{2} \biggr]
		\\
		&\leq \tstep C_1 \sum_{j = 0}^{k - 1} \Expect \biggl[ \max_{\ell \leq j} \norm{ e_{\ell} }^{2} \biggr] + 2TC_\numflow^2\tstep^{2q} + TC_{\xi,R}^2\tstep^{2p}
		\\
		&\quad + \frac{1}{2}\left( \Expect \left[ \max_{k\leq \ell} \norm{e_{k}}^2 \right] + 2C_\numflow^2 \tstep^{1 + 2q}\right) 
		\\
		&\quad+ 36(1 + \tstep^\ast)(1 + C_\odeflow\tstep^\ast)^2T\tstep^{-1} \left(C_{\xi,R}\tstep^{p + 1/2}\right)^2,
	\end{align*}
	and by rearranging terms and using that $\tstep<\tstep^\ast\leq 1$, we obtain
	\begin{align*}
		\Expect& \biggl[ \max_{\ell \leq k} \norm{ e_{\ell} }^{2} \biggr]
		\\
		&\leq 2\tstep C_1 \sum_{j = 0}^{k - 1} \Expect \biggl[ \max_{\ell \leq j} \norm{ e_{\ell} }^{2} \biggr]+ 4(1 + T)C_\numflow^2\tstep^{2q} \\
		&\quad  + 2TC_{\xi,R}^2(1 + 36(1 + \tstep^\ast)(1 + C_\odeflow\tstep^\ast)^2)\tstep^{2p}.
	\end{align*}
	By the discrete Gr\"onwall inequality (Theorem \ref{thm:discrete_Gronwall}) with $x_k\defeq \Expect [ \max_{\ell \leq k} \norm{ e_{\ell} }^{2} ]$ and constant $\alpha_k$ and $\beta_j=2\tstep C_1$, and by using that $K=T/\tstep$, we obtain
	\begin{align*}
		\Expect &\biggl[ \max_{\ell\in[K]} \norm{ e_{\ell} }^{2} \biggr]
		\\
		&\leq\exp( 2 T C_1 )\left[4(1 + T)C_\numflow^2\tstep^{2q}\right. \\
		&\hskip15ex \left.+ 2TC_{\xi,R}^2(1 + 36(1 + \tstep^\ast)(1 + C_\odeflow\tstep^\ast)^2)\tstep^{2p}\right].
	\end{align*}
	This establishes \eqref{eq:strong_error_sup_inside_discrete}.\qed
\end{proof}

\begin{proof}[Proof of Theorem \ref{thm:convergence_without_iid_mean_zero_noise}]
	Let $0\leq k\leq K-1$ and $n \in \Naturals$.
	By applying the triangle inequality, \eqref{eq:gen_peter_paul_inequality}, Assumptions \ref{ass:exact_flow} and \ref{ass:numerical_flow}, and by using that $1 + \tstep 2^{n-1}\leq 1 + 2^{n-1}$ (since $\tstep\leq 1$),
	\begin{align}
		&\norm{e_{k + 1}}^n
		\notag\\
		&\leq \left(\norm{\odeflow^\tstep(u_k)-\numflow^\tstep(U_k)} + \norm{\xi_k(\tstep)}\right)^n
		\notag\\
		&\leq (1 + \tstep 2^{n-1})\norm{\odeflow^\tstep(u_k)-\numflow^\tstep(U_k)}^n \notag\\
		&\quad+ (1 + (2/\tstep)^{n-1})\norm{\xi_k(\tstep)}^n
		\notag\\
		&\leq (1 + \tstep 2^{n-1})\left( (1 + \tstep 2^{n-1}) \norm{\odeflow^\tstep(u_k)-\odeflow^\tstep(U_k)}^n\right. \notag\\
		&\quad+\left. (1 + (2/\tstep)^{n-1})\norm{\odeflow^\tstep(U_k)-\numflow^\tstep(U_k)}^n\right)\notag\\
		&\quad + (1 + (2/\tstep)^{n-1})\norm{\xi_k(\tstep)}^n\notag
		\\
		&\leq (1 + \tstep 2^{n-1})^2(1 + \tstep C_\odeflow)^n\norm{e_k}^n \notag\\
		&\quad+ \left(1 + (2/\tstep)^{n-1}\right)\left((1 + 2^{n-1})C_\numflow^n\tstep^{n(q + 1)} + \norm{\xi_k(\tstep)}^n\right).
		\notag
	\end{align}
	Observe that, since $2^{n-1}$ and $C_\odeflow$ are nonnegative, and since $0<\tau<\tau^\ast$,
	\begin{equation}
		\label{eq:C_odeflow_n}
		C_\odeflow(n,\tstep)\defeq \left[(1 + \tstep 2^{n-1})^2(1 + \tstep C_\odeflow)^n-1\right]\tstep^{-1}.
	\end{equation}
	Note that $C_\odeflow(n,\tstep)\leq C_\odeflow(n,\tstep^\ast)$. 
	
	Since $n\geq 1$ implies that $1 + (2/\tstep)^{n-1}\leq 2^n\tstep^{1-n}$, we have
	\begin{align}
		&\norm{e_{k + 1}}^n-\norm{e_{k}}^n 
		\notag\\
		&\leq C_\odeflow(m,\tstep^\ast)\tstep\norm{e_{k}}^{m} + \tstep^{1-n}(1 + 2^{n-1})^2C^n_\numflow \tstep^{n(q + 1)} 
		\notag
		\\
		&\quad+ \tstep^{1-n}(1 + 2^{n-1})\norm{\xi_k(\tstep)}^n
		\notag\\
		&\leq C_\odeflow(m,\tstep^\ast)\tstep\norm{e_k}^n + \tstep^{1-n}(4 C_\numflow\tstep^{q + 1})^n 
		\notag\\
		&\quad+ \tstep^{1-n}(2\norm{\xi_k(\tstep)})^n.
		\notag
	\end{align}
	Decomposing $\norm{e_{k + 1}}^n-\norm{e_0}^n$ as a telescoping sum, using that $e_0=u_0-U_0=0$, using the nonnegativity of the summands on the right-hand side of the last inequality, and using the relation $\norm{e_\ell}^n\leq \max_{j\leq \ell}\norm{e_j}^n$, we obtain
	\begin{align}
		\max_{\ell\leq k + 1}\norm{e_\ell}^n&\leq \left(\tstep^{1-n}\sum^{K-1}_{k=0}\left((4 C_\numflow\tstep^{q + 1})^n + \left(2\norm{\xi_k(\tstep)}\right)^n\right)\right) 
		\notag
		\\
		&\quad+ C_\odeflow(n,\tstep^\ast)\tstep\sum^k_{\ell=0}\max_{j\leq \ell}\norm{e_j}^n.
		\notag
	\end{align}
	Using that $K=T\tstep$ and Gr\"{o}nwall's inequality (Theorem \ref{thm:discrete_Gronwall}),
	\begin{align}
		\max_{\ell\in[K]}&\norm{e_\ell}^n
		\notag\\
		&\leq (4 C_\numflow\tstep^{q})^n T\exp\left(TC_\odeflow(n,\tstep^\ast)\right) 
		\label{eq:bound_on_nth_power_max_error}
		\\
		&\quad+ \left(\tstep^{1-n}2^n\sum^{K-1}_{k=0}\norm{\xi_k(\tstep)}^n\right)\exp\left(TC_\odeflow(n,\tstep^\ast)\right).
		\notag
	\end{align}
	Taking expectations, using \eqref{eq:upper_bound_expectation_nth_power_sum_norm_xi_sq} with $w=n$ and $v=1$, and using that $K=T/\tstep$ yields
	\begin{align*}
		\Expect\left[\max_{\ell\in[K]}\norm{e_\ell}^n\right] &\leq(4 C_\numflow\tstep^{q})^n T\exp\left(TC_\odeflow(n,\tstep^\ast)\right) 
		\\
		&\quad+ 2^n TC^n_{\xi,R}\tstep^{n(p-1/2)}\exp\left(TC_\odeflow(n,\tstep^\ast)\right).
	\end{align*}
	Rearranging the above produces the desired inequality.\qed
\end{proof}

\begin{proof}[Proof of Corollary \ref{cor:finite_mgf_lipschitz_flow}]
	Let $m\in\Naturals$ be arbitrary.
	Using \eqref{eq:bound_on_nth_power_max_error}, and applying \eqref{eq:gen_triangle_inequality} twice, we obtain
	\begin{align*}
		\max_{\ell\in[K]}\norm{e_\ell}^{nm}&\leq 2^{m-1}e^{(TC_\odeflow(n,\tstep^\ast))^m}\biggr[\left((4 C_\numflow\tstep^q)^{n}T\right)^m \\
		&+\left. (\tstep^{1-n}2^{n})^m\left(\sum^{K-1}_{k=0}\norm{\xi_k(\tstep)}^{n}\right)^m\right].
	\end{align*}
	Taking expectations and using \eqref{eq:upper_bound_expectation_nth_power_sum_norm_xi_sq} with $w=n$ and $v=m$, we obtain
	\begin{align*}
	\Expect&\left[\max_{\ell\in[K]}\norm{e_\ell}^{nm}\right]
	\\
	&\leq 2^{m-1}\exp(TC_\odeflow(n,\tstep^\ast))^m\left(\left((4 C_\numflow\tstep^q)^{n}T\right)^m \right.\\
		&\quad+\left. \left(\tstep^{1-n}2^n\right)^m \left(TC^n_{\xi,R}\tstep^{n(p + 1/2)-1}\right)^m\right).
	\end{align*}
	The conclusion follows by the series expansion of the exponential and the dominated convergence theorem.\qed
\end{proof}

\begin{proof}[Proof of Theorem \ref{thm:strong_error_sup_inside_polynomial_growth_and_bounded_noise}]
	Recall that the solution map $\odeflow^\tstep$ of the initial value problem \eqref{eq:the_ivp} satisfies
	\begin{equation*}
		\odeflow^\tstep(a) \defeq a + \int_0^\tstep f(\odeflow^{t}(a)) \, \rd t.
	\end{equation*}
	For any $\tstep>0$ and $a,b\in\Reals^{d}$, Assumption \ref{ass:polynomial_growth} and the integral Gr{\"o}nwall--Bellman inequality yield
	\begin{align*}
		&\norm{\odeflow^{\tstep}(a)-\odeflow^{\tstep}(b)}\\
		& = \Norm{ a-b + \int_0^{\tstep} f(\odeflow^{t}(a))-f(\odeflow^{t}(b))\, \rd t}
	\\
		&\leq \Norm{a-b}
		\\
		&\quad+ D\int_0^{\tstep} (1 + \norm{\odeflow^{t}(a)}^s + \norm{\odeflow^{t}(b)}^s)\norm{\odeflow^{t}(a)-\odeflow^{t}(b)}\,\rd t
		\\
		& \leq \Norm{a-b}\exp\left(D\int_0^{\tstep} (1 + \norm{\odeflow^{t}(a)}^s + \norm{\odeflow^{t}(b)}^s)\, \rd t\right).
	\end{align*}
	Given the boundedness hypothesis on the $(\xi_k(\tstep))_{k\in[K]}$, we may define a finite constant $C>0$ that does not depend on $\tstep$ or $k$, such that
	\begin{align*}
		\norm{\odeflow^{\tstep}(u_{k}) - \odeflow^{\tstep}(U_{k})} &\leq \norm{e_{k}} \exp\left( D\tau(1 + 2C)\right)\\
		&\leq \norm{e_{k}} (1 + C'\tstep).
	\end{align*}
	The rest of the proof follows in a similar manner to that of Theorem \ref{thm:convergence_without_iid_mean_zero_noise}. \qed
\end{proof}

\begin{proof}[Proof of Lemma \ref{lem:as_bound_on_abs_Un_sq}]
	In what follows, we shall omit the dependence of all random variables on $\omega$, with the understanding that $\omega$ is arbitrary.
	Let $n\in [K]$, where $K = T/\tstep\in\Naturals$.
	From \eqref{eq:randomised_numerical_scheme_0} we have, by \eqref{eq:Young},
	\begin{align}
		\label{eq:sq_abs_U_nplus1}
		\norm{U_{n + 1}}^2 \leq (1 + \tstep) \Norm{\numflow^{\tstep}(U_n)}^2 + (1 + \tstep^{-1}) \norm{\xi_n(\tstep)}^2 .
	\end{align}
	Taking the inner product of \eqref{eq:implicit_euler} with $\numflow^{\tstep}(U_n)$, we obtain by \eqref{eq:dissipativity_property}
	\begin{align*}
		&\Norm{\numflow^{\tstep}(U_n)}^2\\
		& = \innerprod{ \numflow^{\tstep}(U_n) }{ U_n } + \tstep \innerprod{ f(\numflow^{\tstep}(U_n)) }{ \numflow^{\tstep}(U_n) } \\
		& \leq \frac{1}{2} \left( \Norm{\numflow^{\tstep}(U_n)}^2 + \norm{U_n}^2\right) + \tstep \left(\alpha + \beta \Norm{\numflow^{\tstep}(U_n)}^2 \right) \\
		& = \Norm{\numflow^{\tstep}(U_n)}^2 \left(\frac{1}{2} + \beta\tstep\right) + \frac{1}{2} \norm{U_n}^2 + \alpha\tstep .
	\end{align*}
	Thus,
	\begin{align}
		\label{eq:bound_on_sq_abs_psi_tau_U_n}
		\Norm{\numflow^{\tstep}(U_n)}^2
		&\leq \frac{1}{1-2\beta\tstep} \left( \norm{U_n}^2 + 2\alpha\tstep \right)
		\notag
		\\
		&\leq \frac{1}{1-2 \absval{ \beta } \tstep} \left( \norm{U_n}^2 + 2\alpha\tstep \right) ,
	\end{align}
	where we used the inequality $1-2\absval{ \beta }\tstep \leq 1 + 2\beta\tstep$ for the second inequality.
	Then \eqref{eq:sq_abs_U_nplus1} and \eqref{eq:bound_on_sq_abs_psi_tau_U_n} yield
	\begin{equation}
		\label{eq:as_bound_on_abs_Un_sq}
		\norm{U_n}^2 \leq \frac{1 + \tstep}{1-2\absval{ \beta }\tstep} \left( \norm{U_{n-1}}^2 + 2\alpha\tstep \right) + \frac{1 + \tstep}{\tstep} \norm{\xi_{n-1}(\tstep)}^2 .
	\end{equation}
	Let $c_1(\tstep)\defeq \tfrac{1+2\absval{\beta}}{1-2\absval{\beta}\tstep}$ and $c_2(\tstep)\defeq \tfrac{2\alpha}{1-2\absval{\beta}\tstep}$. By \eqref{eq:as_bound_on_abs_Un_sq}, it follows that
	\begin{align*}
		&\norm{U_n}^2 - \norm{U_{n-1}}^2 
		\\
		&\quad\leq \tstep c_1(\tstep) \norm{U_{n-1}}^2
		+ (1 + \tstep) \left(\tstep c_2(\tstep) + \tstep^{-1} \norm{\xi_{n-1}(\tstep)}^2 \right) .
	\end{align*}
	Using the telescoping sum
	\begin{equation*}
		\norm{U_n}^2 = \norm{U_0}^2 + \sum^n_{i=1}\left(\norm{U_i}^2-\norm{U_{i-1}}^2\right)
	\end{equation*}
	it follows that
	\begin{align*}
		\norm{U_n}^2 \leq \norm{U_0}^2 + &\sum^n_{i=1} \left[ \tstep c_1(\tstep)\norm{U_{i-1}}^2 +\right. 
		\\
		&\quad\quad\left.(1 + \tstep) \left( \tstep c_2(\tstep) + \tstep^{-1}\norm{\xi_{i-1}(\tstep)}^2 \right) \right].
	\end{align*}
	Since $n \leq K \defeq T/\tstep$, and since the right-hand side of the inequality above is nonnegative,
	\begin{align*}
		\norm{U_n}^2 \leq &\norm{U_0}^2 + (1 + \tstep) \left( Tc_2(\tstep) + \tstep^{-1}\sum^{T/\tstep}_{i=1}\norm{\xi_{i-1}(\tstep)}^2\right)
		\\
		&+ \tstep c_1(\tstep)\sum^{n-1}_{i=0}\norm{U_i}^2.
	\end{align*}
	Applying the Gr{\"o}nwall inequality (Theorem \ref{thm:discrete_Gronwall}), yields, for all $n\in[K]$,
	\begin{align*}
		&\max_{i\in [K]} \norm{U_i}^2
		\\
		& \leq \left[\norm{U_0}^2 + (1 + \tstep) \left(Tc_2 + \frac{1}{\tstep}\sum^{T/\tstep}_{i=1} \norm{\xi_{i-1}(\tstep)}^2 \right) \right] \exp \left( Tc_1 \right) \\
		& \leq (1 + \tstep) \left[\norm{U_0}^2 + Tc_2 + \frac{1}{\tstep} \sum^{T/\tstep}_{i=1} \norm{\xi_{i-1}(\tstep)}^2\right] \exp \left( Tc_1\right)\\
		& \leq C_2 \left( 1 + \tstep^{-1}\sum^{T/\tstep}_{i=1}\norm{\xi_{i-1}(\tstep)}^2 \right),
	\end{align*}
	where we define, for $\tstep'$ as in Assumption \ref{ass:tau}, the scalar
	\begin{equation}
		\label{eq:C_2}
		C_2=(1 + \tstep')\max\left\{1, \norm{ U_0 }^2 + Tc_2(\tstep')\right\}\exp\left(Tc_1(\tstep')\right).
	\end{equation}
	This yields \eqref{eq:abs_Un} for $n=1$.
	By applying \eqref{eq:gen_triangle_inequality}, we obtain \eqref{eq:abs_Un} for arbitrary $n\in\Naturals$.\qed
\end{proof}

\begin{proof}[Proof of Proposition \ref{prop:W_k}]
	Recall that in Assumption \ref{ass:polynomial_growth}, we assume $f \in C^1(\Reals^{d}; \Reals^{d})$.
	Therefore, Taylor's theorem applied to the function $t\mapsto \odeflow^t(a)$ yields
	\begin{equation}
		\label{eq:taylor_remainder_term}
		\odeflow^\tstep(a)=a + \tstep f(a) + \tstep R^\tstep(a),
	\end{equation}
	where $R^\tstep(a)\to 0$ as $\tstep\to 0$.
	Then, by \eqref{eq:polynomial_growth_of_vector_field}, \eqref{eq:implicit_euler}, and \eqref{eq:gen_triangle_inequality},
	\begin{align}
	&\norm{\numflow^\tstep(U_k)-\odeflow^\tstep(U_k)}^{2n}
	\notag
	\\
	&\leq 2^{2n-1}\tstep^{2n} \left(\norm{ f(\numflow^\tstep(U_k))- f(U_k)	}^{2n}+  \norm{R^\tstep(U_k)}^{2n}\right).
	\label{eq:temp01}
	\end{align}
	By \eqref{eq:polynomial_growth_of_vector_field}, \eqref{eq:implicit_euler}, \eqref{eq:bound_on_norm_f_of_a}, and \eqref{eq:gen_triangle_inequality} with the fact that $C_\odeflow\geq 1$ in Assumption \ref{ass:polynomial_growth}, we obtain
	\begin{align*}
		&\Norm{f(U_{k})-f(\numflow^{\tstep}(U_{k}))}^{2n} \\
		& \leq C_\odeflow^{2n} \left( 1 + \norm{U_{k}}^{s} + \Norm{\numflow^{\tstep}(U_{k})}^{s} \right)^{2n} \Norm{U_{k}-\numflow^{\tstep}(U_{k})}^{2n} \\
		& = C_\odeflow^{2n} \left( 1 + \norm{U_{k}}^{s} + \Norm{\numflow^{\tstep}(U_{k})}^{s}\right)^{2n} \Norm{\tstep f(\numflow^{\tstep}(U_{k}))}^{2n}
		\\
		&\leq \tstep^{2n} C_\odeflow^{2n}\left( 1 + \norm{U_{k}}^{s} + \Norm{\numflow^{\tstep}(U_{k})}^{s}\right)^{2n}\times 
		\\
		&\quad\left(C_\odeflow\left(1 + \Norm{ \numflow^{\tstep}(U_{k})}^s\right)\norm{\numflow^\tstep(U_k)} + \norm{f(0)}\right)^{2n}
		\\
		&\leq \tstep^{2n} 3^{2(2n-1)}C_\odeflow^{4n}\left( 1 + \norm{U_{k}}^{2ns} + \Norm{\numflow^{\tstep}(U_{k})}^{2ns}\right)\times
		\\
		&\quad\left(\norm{\numflow^\tstep(U_k)}^{2n} + \Norm{ \numflow^{\tstep}(U_{k})}^{2n(s + 1)} + \norm{f(0)}^{2n}\right).
	\end{align*}
	From \eqref{eq:bound_on_sq_abs_psi_tau_U_n} and \eqref{eq:gen_triangle_inequality}, it holds that for any $n$ and $r$ such that $nr\geq 1$,
	\begin{align*}
		\norm{\numflow^\tstep(U_k)}^{2nr}&\leq \frac{2^{nr-1}}{(1-2\absval{\beta}\tstep)^{nr}}\left(\norm{U_k}^{2nr} + (2\alpha\tstep)^{nr}\right)
		\\
		&\leq \frac{2^{nr-1}}{(1-2\absval{\beta}\tstep')^{nr}}\left(\norm{U_k}^{2nr} + (2\alpha\tstep')^{nr}\right),
	\end{align*}
	for $\tstep'$ in Assumption \ref{ass:tau}.
	Applying the second inequality for the appropriate values of $r$ and computing exponents yields that, for the polynomials $\pi_1$, $\pi_2$ and $\pi$ defined on $\Reals$ by
	\begin{align*}
		\pi_1(x)&\defeq \left( 1 + x^{ns} + \frac{2^{ns-1}}{(1-2\absval{\beta}\tstep')^{ns}} \left(x^{ns} + \left(2\alpha\tstep'\right)^{ns}\right)\right)
		\\
		\pi_2(x)&\defeq\frac{2^{n(s + 1)-1}}{(1-2\absval{\beta}\tstep')^{n(s + 1)}} \times
		\\
		&\left( x^{n} + \left(2\alpha\tstep'\right)^{n} + x^{n(s + 1)}+ \left(2\alpha\tstep'\right)^{n(s + 1)} + \norm{f(0)}^{2n}\right)
	\end{align*}
	and $\pi(x)\defeq \pi_1(x)\pi_2(x)$, it follows from Lemma  \ref{lem:as_bound_on_abs_Un_sq} that
	\begin{align*}
		&\Norm{f(U_{k})-f(\numflow^{\tstep}(U_{k}))}^{2n}
		\\
		&\quad\leq \tstep^{2n}3^{2(2n-1)}C^{4n}_\odeflow\pi\left(\norm{U_k}^2\right)
		\\
		&\quad\leq \tstep^{2n}3^{2(2n-1)}C^{4n}_\odeflow\pi\left(\max_{i\in[K]}\norm{U_i}^2\right).
	\end{align*}
	Taking expectations, applying Proposition \ref{prop:bounds_moments_max_num_soln}, and using that $\tstep<\tstep'$ to bound the right-hand side of the inequality \eqref{eq:Eval_max_i_abs_Ui} in Proposition \ref{prop:bounds_moments_max_num_soln}, we may define some $C_3=C_3(\alpha,\beta,C_\odeflow,\tstep',n)$ that does not depend on $k$ or $\tstep$, such that
	\begin{align}		
		\Expect&\left[\Norm{f(U_{k})-f(\numflow^{\tstep}(U_{k}))}^{2n}\right]
		\notag
		\\
		&\leq \tstep^{2n} 3^{2(2n-1)}C_\odeflow^{4n}\Expect\left[\pi\left(\max_{i\in[K]}\norm{U_i}^2\right)\right]\qefed \tstep^{2n}C_3.
		\label{eq:C_3}
	\end{align}
	By Proposition \ref{prop:bounds_moments_max_num_soln}, the finiteness of $C_3$ follows from the hypothesis $R\geq 2n(2s + 1)$ and the observation that $\pi_1(x^2)$ and $\pi_2(x^2)$ have degree $ns$ and $n(s + 1)$ in $x^2$, respectively.

	Now it remains to show that $\norm{R^\tstep(U_k)}^{2n}\in L^1_{\mathbb{P}}$.
	From \eqref{eq:polynomial_growth_of_vector_field}, \eqref{eq:polynomial_growth_of_flow_map}, and \eqref{eq:taylor_remainder_term}, we obtain
	\begin{align}
		\tstep &\norm{R^\tstep(a)-R^\tstep(b)}
		\notag
		\\
		&=\norm{\odeflow^\tstep(a)-a-\tstep f(a)-\odeflow^\tstep(b)-b-\tstep f(b)}
		\notag
		\\
		&\leq \norm{\odeflow^\tstep(a)-\odeflow^\tstep(b)}+\norm{a-b}+\tstep\norm{f(a)-f(b)}
		\notag
		\\
		&= 2\left(1 + \tstep C_\odeflow\left(1 + \norm{a}^s + \norm{b}^s\right)\norm{a-b}\right)
		\label{eq:polynomial_growth_of_remainder}
	\end{align}
	By the triangle inequality and \eqref{eq:polynomial_growth_of_remainder},
	\begin{align*}
		&\tstep\norm{R^\tstep(U_k)}&
		\\
		&\leq \tstep\left(\norm{R^\tstep(0)} + 2\left(1 + \tstep C_\odeflow (1 + \norm{U_k}^s)\norm{U_k}\right)\right)
		\\
		&\leq \tstep\left( \norm{R^\tstep(0)} + 2\left(1 + \tstep C_\odeflow (\norm{U_k} + \norm{U_k}^{s + 1})\right)\right)
		\\
		&\leq \tstep\biggr( \norm{R^\tstep(0)} + 2 \\
		&\hskip15ex\left.+ 2C_\odeflow \left(\max_{k\in[K]}\norm{U_k} + \max_{k\in[K]}\norm{U_k}^{s + 1}\right)\right).
	\end{align*}
	Then by applying \eqref{eq:gen_triangle_inequality} and Proposition \ref{prop:bounds_moments_max_num_soln} with the hypothesis that $R\geq 2n(2s + 1)\geq 2n(s + 1)$, and using the bound $\tstep<\tstep'$, it follows that we may define a positive scalar $C_4$ 
	that does not depend on $k$ or $\tstep$, such that
	\begin{align}
		\Expect&\left[\tstep^{2n}\norm{R^\tstep(U_k)}^{2n} \right]
		\notag
		\\
		&\leq \tstep^{2n}\Expect\biggr[\biggr( \norm{R^\tstep(0)} + 2 +
		\notag\\
		&\quad\quad\left.\left. 2C_\odeflow \left(\max_{k\in[K]}\norm{U_k} + \max_{k\in[K]}\norm{U_k}^{s + 1}\right)\right)^{2n}\right]
		\notag\\
		&\qefed \tstep^{2n}C_4.
		\label{eq:C_4}	
	\end{align}
	Therefore, with $C_3$ and $C_4$ as in \eqref{eq:C_3} and \eqref{eq:C_4} above, \eqref{eq:temp01} yields
	\begin{align}		
		\Expect\left[ \norm{\numflow^\tstep(U_k)-\odeflow^\tstep(U_k)}^{2n}\right]&\leq 2^{2n-1}\tstep^{4n}\left(C_3 + C_4\right)
		\notag
		\\
		\qefed C_\numflow\tstep^{4n}
		\label{eq:C_numflow}
	\end{align}
	as desired.\qed
\end{proof}

The proof below makes clear that we make absolutely no effort to find optimal constants.
\begin{proof}[Proof of Theorem \ref{thm:convergence_without_iid_mean_zero_noise_implicit_euler}]
	Let $n\in\Naturals$.
	By \eqref{eq:gen_peter_paul_inequality}
	\begin{align}
		&\norm{e_{k + 1}}^{2n}
		\notag
		\\
		&\leq (1 + \tstep 2^{2n-1})\left[(1 + \tstep 2^{2n-1})\norm{\odeflow^\tstep(u_k)-\odeflow^\tstep(U_k)}^{2n}\right.
		\notag
		\\
		&\quad\left. + (1 + (2/\tstep)^{2n-1})\norm{\odeflow^\tstep(U_k)-\numflow^\tstep(U_k)}^{2n}\right] 
		\notag
		\\
		&\quad+ (1 + (2/\tstep)^{2n-1}\norm{\xi_k(\tstep)}^{2n}.
		\notag
	\end{align}
	Since $\tstep\leq 1$ and $n\geq 1$, it holds that $1 + \tstep^{1-2n} 2^{2n-1}\leq\tstep^{1-2n}(1 + 2^{2n-1})$ and $1 + \tstep 2^{2n-1}\leq 1 + 2^{2n-1}$.
	Using these inequalities, \eqref{eq:gen_peter_paul_inequality}, and \eqref{eq:polynomial_growth_of_flow_map} in the preceding inequality, we obtain
	\begin{align*}
		&\norm{e_{k + 1}}^{2n}
		\\
		&\leq (1 + \tstep 2^{2n-1})^2\left(1 + \tstep C_\odeflow(1 + \norm{u_k}^s + \norm{U_k}^s)\right)^{2n}\norm{e_k}^{2n}
		\\
		&\quad+ (1 + 2^{2n-1})^2\tstep^{1-2n}\times
		\\
		&\quad\quad\left( \norm{\odeflow^\tstep(U_k)-\numflow^\tstep(U_k)}^{2n} + \norm{\xi_k(\tstep)}^{2n}\right).		
	\end{align*}
	Using \eqref{eq:gen_peter_paul_inequality} again, we obtain
	\begin{align*}
		&\left(1 + \tstep C_\odeflow(1 + \norm{u_k}^s + \norm{U_k}^s)\right)^{2n}\norm{e_k}^{2n}
		\\
		&\quad\leq \left[\left(1 + \tstep C_\odeflow \left(1 + \max_{\ell\in[K]}\norm{u_\ell}^s\right)\right)\norm{e_k} +\right. 
		\\
		&\hskip30ex\left.\tstep C_\odeflow\max_{\ell\in[K]}\norm{U_\ell}^s\norm{e_k}\right]^{2n}
		\\
		&\quad\leq (1 + \tstep 2^{2n-1})\left(1 + \tstep C_\odeflow \left(1 + \max_{\ell\in[K]}\norm{u_\ell}^s\right)\right)^{2n}\norm{e_k}^{2n}
		\\
		&\quad\quad + (1 + \tstep^{1-2n}2^{2n-1})(\tstep C_\odeflow)^{2n}\max_{\ell\in[K]}\norm{U_\ell}^{2ns}\norm{e_k}^{2n}
	\end{align*}
	so that by defining 
	\begin{equation*}
	 C_5=C_5 (n,s,C_\odeflow,u)\defeq \max\{2^{2n-1},C_\odeflow(1 + \norm{u}^s_\infty)\}
	\end{equation*}
	we have
	\begin{align*}
		&\left(1 + \tstep C_\odeflow(1 + \norm{u_k}^s + \norm{U_k}^s)\right)^{2n}\norm{e_k}^{2n}
		\\
		&\quad \leq \left(1 + \tstep C_5\right)^{2n + 1}\norm{e_k}^{2n} 
		\\
		&\quad\quad+ \tstep^{1-2n}(1 + 2^{2n-1})\left(\tstep C_\odeflow\max_{\ell\in[K]}\norm{U_\ell}^{s}\norm{e_k}\right)^{2n}
	\end{align*}
	and, therefore,
	\begin{align*}
		&\norm{e_{k + 1}}^{2n}-\norm{e_k}^{2n}
		\\
		&\leq\left[\left(1 + \tstep C_5\right)^{2n + 1}-1\right]\tstep^{-1}\tstep\norm{e_k}^{2n} 
		\\
		&\ + \tstep^{1-2n}(1 + 2^{2n-1})\left(\tstep C_\odeflow\max_{\ell\in[K]}\norm{U_\ell}^{s}\norm{e_k}\right)^{2n}
		\\
		&\ + (1 + 2^{2n-1})^2\tstep^{1-2n}\times
		\\
		&\hskip20ex\left( \norm{\odeflow^\tstep(U_k)-\numflow^\tstep(U_k)}^{2n} + \norm{\xi_k(\tstep)}^{2n}\right).
	\end{align*}
	By nonnegativity of $C_5$, it follows that $[(1 + \tstep C_5)^{2n + 1}-1]\tstep^{-1}$ is a polynomial of degree $2n$ in $\tstep$ with positive coefficients.
	In particular, if we recall the definition of $C_5$ and define $C_6$ by
	\begin{align}
		&C_6(C_\odeflow,n,s,(u_t)_{0\leq t\leq T},\tstep')
		\notag\\
		&\defeq\left[\left(1 + \tstep' \max\{2^{2n-1},C_\odeflow(1 + \norm{u}^s_{\infty}\}\right)^{2n + 1}-1\right](\tstep')^{-1} ,
		\label{eq:C_6}
	\end{align}
	then $C_6$ does not depend on $\tstep$, $[(1 + \tstep C_5)^{2n + 1}-1]\tstep^{-1}\leq C_6$ for all $0<\tstep<\tstep'$, and
	\begin{align*}
		&\norm{e_{k + 1}}^{2n}-\norm{e_k}^{2n}
		\\
		&\leq C_6\tstep\norm{e_k}^{2n}
		\\
		&\quad+ \tstep^{1-2n}(1 + 2^{2n-1})\left(\tstep C_\odeflow\max_{\ell\in[K]}\norm{U_\ell}^{s}\norm{e_k}\right)^{2n}
		\\
		&\quad\quad + (1 + 2^{2n-1})^2\tstep^{1-2n}\times
		\\
		&\hskip20ex\left( \norm{\odeflow^\tstep(U_k)-\numflow^\tstep(U_k)}^{2n} + \norm{\xi_k(\tstep)}^{2n}\right).
	\end{align*}
	By the telescoping sum associated to $\norm{e_{k + 1}}^{2n}-\norm{e_k}^{2n}$, the fact that $e_0=0$, the bound $1 + 2^{2n-1}\leq 2^{2n}$, the nonnegativity of the terms on the right-hand side of the inequality above, and the bound $\norm{e_j}\leq \max_{\ell\leq j}\norm{e_\ell}$, we obtain
	\begin{align*}
		&\max_{\ell\leq k + 1}\norm{e_{\ell}}^{2n}
		\\
		&\leq\tstep^{1-2n}2^{4n} \sum^{K}_{\ell=1}\left(\left(\tstep C_\odeflow\max_{j\in[K]}\norm{U_j}^{s}\norm{e_\ell}\right)^{2n} +\right. 
		\\ &\hskip25ex \norm{\odeflow^\tstep(U_\ell)-\numflow^\tstep(U_\ell)}^{2n} + \norm{\xi_\ell(\tstep)}^{2n}\biggr)
		\\
		&\quad + C_6\tstep\sum^{k}_{\ell=1} \max_{j\leq \ell}\norm{e_j}^{2n}.
	\end{align*}
	By Lemma \ref{lem:as_bound_on_abs_Un_sq},
	\begin{align*}
		&\left(\tstep C_\odeflow\max_{j\in[K]}\norm{U_j}^{s}\norm{e_\ell}\right)^{2n}
		\\
		&\leq C_\odeflow^{2n}(2C_1)^{sn}\times
		\\
		&\quad\left(\tstep^{2n}\norm{e_\ell}^{2n} + \tstep^{n(2-s)}\left(\sum^{T/\tstep}_{i=1}\norm{\xi_i(\tstep)}^2\right)^{ns}\norm{e_\ell}^{2n}\right)
	\end{align*}
	which implies that
	\begin{align*}
	&\max_{\ell\leq k + 1}\norm{e_{\ell}}^{2n}
	\\
	&\leq \tstep^{1-ns}2^{n(4 + s)}C_\odeflow^{2n}C_1^{ns} \sum^{K}_{\ell=1}\left(\sum^{T/\tstep}_{i=1}\norm{\xi_i(\tstep)}^2\right)^{ns}\norm{e_\ell}^{2n}
		\\
		&\quad + \tstep^{1-2n}2^{4n} \sum^{K}_{\ell=1}\left(\norm{\odeflow^\tstep(U_\ell)-\numflow^\tstep(U_\ell)}^{2n} + \norm{\xi_\ell(\tstep)}^{2n}\right)
		\\
		&\quad + \left(C_6\tstep + \tstep^{2n}2^{n(4 + s)}C_\odeflow^{2n}C_1^{ns}\right)\sum^{k}_{\ell=1} \max_{j\leq \ell}\norm{e_j}^{2n}.
	\end{align*}
	Define
	\begin{equation}
		\label{eq:C_7}
		C_7=C_7(n,s,C_\odeflow,C_1)\defeq 2^{n(4 + s)}C_\odeflow^{2n}C_1^{ns}.
	\end{equation}
	Since $C_\odeflow,C_1\geq 1$, it follows that $2^{4n}\leq C_7$ ,and by Gr\"{o}nwall's inequality (Theorem \ref{thm:discrete_Gronwall}) we obtain
	\begin{align*}
		&\max_{\ell\in [K]}\norm{e_{\ell}}^{2n}
		\\
		&\leq \exp\left(T(C_6 + \tstep^{2n-1}C_7)\right)C_7\times
		\\
		&\quad\left( \tstep^{1-ns}\sum^K_{\ell=1}\left(\sum^{T/\tstep}_{i=1}\norm{\xi_i(\tstep)}^2\right)^{ns}\norm{e_\ell}^{2n}\right.
		\\
		&\hskip7ex + \left.\tstep^{1-2n}\sum^K_{\ell=1}\left(\norm{\odeflow^\tstep(U_\ell)-\numflow^\tstep(U_\ell)}^{2n} + \norm{\xi_\ell(\tstep)}^{2n}\right)\right).
	\end{align*}
	Taking expectations completes the proof, provided that we can ensure each sum is of the right order in $\tstep$.
	By Proposition \ref{prop:W_k} with the hypothesis that $R\geq 2n(2s + 1)$, and by Assumption \ref{ass:p-R_regularity_condition_on_noise},
	\begin{align}
		&\tstep^{1-2n} \sum^K_{\ell=1} \Expect\left[\norm{\odeflow^\tstep(U_\ell) - \numflow^\tstep(U_\ell)}^{2n} + \norm{\xi_\ell(\tstep)}^{2n}\right]
		\notag\\
		&\leq T\left(C_\numflow\tstep^{2n} + \left(C_{\xi,R}\tstep^{p-1/2}\right)^{2n}\right).
		\label{eq:first_condition_on_p_inequality}
	\end{align}
	Thus we need $p-1/2\geq 1$ to hold.
	Next, using the bound $\norm{e_{\ell}}\leq \max_{j\in[K]}\norm{e_j}$, Young's inequality \eqref{eq:Young} with $a=(\sum^{K}_{i=1}\norm{\xi_{i}(\tstep)}^2)^{ns}$, $b=\norm{e_{\ell}}^{2n}$, and some $\delta>0$ and conjugate exponent pair $(r,r^\ast)\in (1,\infty)^2$ to be determined later, we obtain with \eqref{eq:upper_bound_expectation_nth_power_sum_norm_xi_sq} that
	\begin{align*}
		&\Expect\left[\left(\sum^{T/\tstep}_{i=1} \norm{\xi_{i}(\tstep)}^{2}\right)^{ns} \norm{e_{\ell}}^{2n}\right]
		\\
		&\leq \frac{\delta}{r} \Expect\left[\left(\sum^{T/\tstep}_{i=1} \norm{\xi_{i}(\tstep)}^2\right)^{nrs} \right] + \frac{1}{\delta^{r^\ast/r}r^\ast}\Expect\left[\max_{\ell\in[K]} \norm{e_{\ell}}^{2nr^\ast}\right]
		\\
		&\leq \frac{\delta}{r} \left(TC^2_{\xi,R}\tstep^{2p}\right)^{nrs} + \frac{1}{\delta^{r^\ast/r}r^\ast}\Expect\left[\max_{\ell\in[K]} \norm{e_{\ell}}^{2nr^\ast}\right].
	\end{align*}
	Since $R\geq 2n(2s + 1)$, the maximal value of $r$ compatible with integrability of $(\sum^{K}_{i=1}\norm{\xi_{i}(\tstep)}^2)^{nrs}$ is $r=2 + s^{-1}$.
	Since we are not interested in optimal estimates, we shall set $r=r^\ast=2$ and $\delta=\tstep^{-n(2 + s)}$.
	We thus obtain
	\begin{align*}
		&\tstep^{1-ns}\sum^K_{\ell=1}~\Expect\left[\left(\sum^{T/\tstep}_{i=1} \norm{\xi_{i}(\tstep)}^{2}\right)^{ns} \norm{e_{\ell}}^{2n}\right]
		\\
		&\leq \frac{T}{2}\tstep^{-ns} \left((TC^2_{\xi,R})^{nrs}\tstep^{-n(2 + s) + 2p(2ns)} \right.
		\\
		&\hskip30ex\left.+ \tstep^{n(2 + s)}\Expect\left[\max_{\ell\in[K]}\norm{e_\ell}^{4n}\right] \right).
	\end{align*}
	For the exponent of $\tstep$ of the first term in the parentheses, we want to ensure that $-n(2 + s) + 2p(2ns)-ns\geq 2n$, or equivalently that $p\geq \tfrac{1}{s} + \tfrac{1}{2}$.
	Comparing this condition with the condition $p-\tfrac{1}{2}\geq 1$ that arose from \eqref{eq:first_condition_on_p_inequality}, and recalling that $s\geq 1$, we observe that if $p\geq \tfrac{3}{2}$, then the preceding estimates yield
	\begin{align*}
		&\Expect\left[\max_{\ell\in [K]}\norm{e_{\ell}}^{2n}\right]
		\\
		&\leq \exp\left(T(C_6 + \tstep^{2n-1}C_7)\right)\frac{C_7 T}{2}\times
		\\
		&\hskip20ex\left( \left(TC^2_{\xi,R}\right)^{2ns} + \Expect\left[\max_{\ell\in[K]}\norm{e_\ell}^{4n}\right]\right)\tstep^{2n}.
	\end{align*}
	It remains to bound $\Expect[\max_{\ell\in[K]}\norm{e_\ell}^{4n}]$ by a constant that does not depend on $\tstep$.
	By \eqref{eq:gen_triangle_inequality}, Proposition \ref{prop:bounds_moments_max_num_soln}, and the assumption that $\tstep<\tstep'$ for $\tstep'$ in Assumption \ref{ass:tau}, we obtain
	\begin{align}
		\Expect&\left[ \max_{\ell\in[K]}\norm{e_\ell}^{4n}\right]
		\notag
		\\
		&\leq 2^{4n}\left(\max_{k\in[K]}\norm{u_k}^{4n} + \Expect\left[\max_{k\in[K]}\norm{U_k}^{4n}\right]\right)
		\notag\\
		&\leq 2^{4n}\left(\norm{u}_{\infty}^{4n} + (2C_2)^{2n}\left(1 + TC_{\xi,R}^2\tstep^{2p-1}\right)^{2n}\right)
		\notag\\
		& \leq 2^{4n}\left(\norm{u}_{\infty}^{4n} + (2C_2)^{2n}\left(1 + TC_{\xi,R}^2(\tstep')^{2p-1}\right)^{2n}\right)\qefed C_8,
		\label{eq:C_8}
	\end{align}
	where $C_8=C_8(C_2,C_{\xi,R},n,p,\tstep',T,(u_t)_{0\leq t\leq T})>0$ does not depend on $\tstep$.
	Note that in applying Proposition \ref{prop:bounds_moments_max_num_soln}, we have used that $s\geq 1$ for the exponent $s$ in Assumption \ref{ass:polynomial_growth}, since this implies that $2n(2s + 1)\geq 4n$. \qed
\end{proof}

\begin{proof}[Proof of Theorem \ref{thm:strong_error_sup_inside_continuous}]
	Let $k\in [K]$ and $t_{k}<t\leq t_{k + 1}$.
	Then
	\begin{align*}
		e(t)& = \odeflow^{t - t_{k}}(u_{k}) - \numflow^{t - t_{k}}(U_{k}) - \xi_{k}(t - t_{k}) \\
		& = \odeflow^{t - t_{k}}(u_{k}) - \odeflow^{t - t_{k}}(U_{k}) + \odeflow^{t - t_{k}}(U_{k}) 
		\\
		&\quad- \numflow^{t - t_{k}}(U_{k}) - \xi_{k}(t - t_{k}) ,
	\end{align*}
	and given that Assumption~\ref{ass:exact_flow} implies that $\odeflow^{t'}$ is Lipschitz on $\Reals^{d}$ for every $t'\geq 0$,
	\begin{align*}
		\norm{ e(t) }^{n}
		&\leq 3^{n-1}\left( ( 1 + C_\odeflow (t - t_{k}) )^{n} \norm{ e_{k} }^{n} \right. 
		\\
		&\hskip15ex\left.+ \left(C_\numflow(t - t_{k})^{q + 1}\right)^{n} + \norm{ \xi_{k}(t - t_{k}) }^{n}\right)
	\end{align*}
	by applying \eqref{eq:gen_triangle_inequality}.
	Since $t-t_{k}\leq \tstep$, it follows from the inequality above that
	\begin{align*}
		\Expect &\biggl[ \sup_{0 \leq t \leq T} \norm{ e(t) }^{n} \biggr]
		\\
		& = \Expect \biggl[ \max_{k\in [K]} \sup_{t_{k} <t \leq t_{k + 1}} \norm{ e(t) }^{n} \biggr] \\
		&\leq 3^{n-1}\left((1 + C_\odeflow\tstep)^n\Expect\left[\max_{k\in[K]}\norm{e_{k}}^{n}\right]+ \left(C_\numflow\tstep^{q + 1}\right)^n\right.
		\\
		&\quad\left.  + \Expect \left[ \max_{0 \leq k \leq K-1} \sup_{t_{k} <t \leq t_{k + 1}} \norm{ \xi_{k}(t - t_{k}) }^{n} \right]\right).
	\end{align*}
	By Assumption \ref{ass:p-R_regularity_condition_on_noise_continuous_time},
	\begin{align*}
		\Expect &\left[ \max_{0 \leq k \leq K-1} \sup_{t_{k} < t \leq t_{k + 1}} \norm{ \xi_{k}(t - t_{k}) }^{n} \right]
		\\
		&\leq\sum_{k = 0}^{K-1}\Expect \left[ \sup_{t_{k}< t \leq t_{k + 1}} \norm{ \xi_{k}(t - t_{k}) }^{n} \right]
		\\
		&\leq \frac{T}{\tstep}\left(C_{\xi,R} \tstep^{p + 1/2}\right)^{n}=TC^n_{\xi,R}\tstep^{n(p + 1/2)-1}
		\\
		&\leq T C^n_{\xi,R}\tstep^{n(p-1/2)}.
	\end{align*}
	Note that Assumption \ref{ass:p-R_regularity_condition_on_noise_continuous_time} is stronger than Assumption \ref{ass:p-R_regularity_condition_on_noise}.
	Therefore we may apply Theorem \ref{thm:convergence_without_iid_mean_zero_noise} to obtain \eqref{eq:strong_error_sup_inside_continuous}.\qed
\end{proof}

\begin{proof}[Proof of Lemma~\ref{lem:rth_power_of_scaled_integrated_BM}]
	If $r=0$, then the desired statement follows immediately.
	Therefore, let $p,r \geq 1$.
	Let $\xi_0$ be the integrated $\mathbb{P}$-Brownian motion process scaled by $\tstep^{p-1}$, so that
	\begin{align*}
		\norm{\xi_0(t)}^r
		& = \tstep^{r(p-1)}t^r \Norm{ \frac{1}{t}\int_0^t B_s \, \rd s }^r 
		\\
		&\leq \tstep^{r(p-1)}t^r\left(\frac{1}{t}\int_0^t \norm{B_s}^r \, \rd s\right) 
		\\
		&= \tstep^{r(p-1)}t^{r-1}\int_0^t \norm{B_s}^r \, \rd s,
	\end{align*}
	where we applied Jensen's inequality to the uniform probability measure on $[0,t]$.
	It follows that
	\begin{align*}
		\Expect &\biggl[ \sup_{t \leq \tstep} \norm{\xi_0(t)}^r \biggr]
		\\
		& \leq \tstep^{r(p-1)}\Expect \biggl[ \sup_{t \leq \tstep}t^{r-1}\int_0^t \norm{B_s}^r \, \rd s \biggr]
		\\
		&\leq \tstep^{r(p-1)}\tstep^{r-1}\Expect \biggl[ \sup_{t \leq \tstep}\int_0^t \norm{B_s}^r \, \rd s \biggr] \\
		& \leq \tstep^{r p-1} \int_0^\tstep \Expect \biggl[ \sup_{0 \leq t \leq \tstep} \norm{B_t}^r \biggr] \, \rd s
		\\
		&= \tstep^{rp} \Expect \biggl[ \sup_{0 \leq t \leq \tstep} \norm{B_t}^r \biggr].
	\end{align*}
	Above, we used the Fubini--Tonelli theorem to interchange expectation and integration with respect to $s$, and the fact that $\Expect \bigl[ \sup_{t \leq \tstep} \norm{B_t}^r \bigr]$ is constant with respect to the variable of integration $s$.
	For $r=1$, the Burkholder--Davis--Gundy martingale inequality \citep[Equation (2.2)]{Peskir:1996} yields
	\begin{equation*}
		\Expect\left[\sup_{0\leq t\leq \tstep}\norm{B_t}^r\right]\leq \frac{4-r}{2-r}\tstep^{r/2},
	\end{equation*}
	with $(4-r)/(2-r)=3$ for $r=1$.
	For $r>1$, Doob's inequality \citep[Equation (2.1)]{Peskir:1996} yields
	\begin{equation*}
		\Expect\left[\sup_{0\leq t\leq \tstep}\norm{B_t}^r\right]\leq \left(\frac{r}{r-1}\right)^r\tstep^{r/2}.
	\end{equation*}
	Since $r\mapsto [r/(r-1)]^r$ is continuously differentiable and monotonically decreasing on $2<r<\infty$, the desired conclusion follows.\qed
\end{proof}

%
%

\end{document}